\documentclass[11pt]{amsart}
\usepackage[latin1]{inputenc}
\usepackage{amsmath}
\usepackage{amsfonts}
\usepackage{amssymb}
\usepackage{graphicx}
\usepackage{fourier}
\usepackage{dsfont}
\usepackage{hyperref}
\usepackage{enumerate}
\usepackage{esint}
\usepackage{bm}
\usepackage{xcolor} % only for highlighting comment \anton
\usepackage{verbatim}        % for comments (Anton)

\usepackage{soul}

\newtheorem{theorem}{Theorem}[section]
\newtheorem{lemma}[theorem]{Lemma}
\newtheorem{prop}[theorem]{Proposition}
\newtheorem{remark}[theorem]{Remark}
\theoremstyle{definition}
\newtheorem{definition}[theorem]{Definition}

\newtheorem{corollary}[theorem]{Corollary}

\theoremstyle{remark}

\newcommand{\norm}[1]{\lVert#1\rVert}
\newcommand{\Norm}[1]{\left\lVert#1\right\rVert}
\newcommand{\abs}[1]{\left\lvert#1\right\rvert}
\newcommand{\pa}[1]{\left( #1 \right)}
\newcommand{\br}[1]{\left\lbrace #1\right\rbrace}

\newcommand{\R}{\mathbb{R}}

\newcommand{\N}{\mathbb{N}}

\numberwithin{equation}{section}

\newcommand{\amit}{\textcolor{black}} 

\newcommand{\D}{{\bf D}}

\newcommand{\II}{{\bf I}}
\newcommand{\C}{{\bf C}}
\newcommand{\B}{{\bf B}}
\newcommand{\A}{{\bf A}}
\newcommand{\K}{{\bf K}}
\newcommand{\M}{{\bf M}}
\newcommand{\Q}{{\bf Q}}

\newcommand{\W}{{\bf W}}

\newcommand{\rank}{\operatorname{rank}}
\newcommand{\tr}{\operatorname{Tr}}
\renewcommand{\span}{\operatorname{span}}
\newcommand{\range}{\operatorname{Range}}
\renewcommand{\Re}{\operatorname{Re}}
\title[Degenerate and Defective Fokker-Planck Equations]{On The Rates of Decay to Equilibrium in Degenerate and Defective Fokker-Planck Equations}
\author{Anton Arnold, Amit Einav \& Tobias W\"ohrer}
\address{Vienna University of Technology, Institute of Analysis and Scientific Computing, Wiedner Hauptstr. 8-10, A-1040 Wien, Austria}
\email{anton.arnold@tuwien.ac.at; aeinav@asc.tuwien.ac.at;}
\email{tobias.woehrer@tuwien.ac.at}
\thanks{The first author was partially supported by the FWF-funded SFB \#F65. 
The second author was supported by the Austrian Science Fund (FWF) grant M 2104-N32.
The first and the third authors were partially supported by the FWF-doctoral school ``Dissipation
and dispersion in nonlinear partial differential equations''.}


\begin{document}

\begin{abstract}We establish sharp long time asymptotic behaviour for a family of entropies to defective Fokker-Planck equations and show that, much like defective finite dimensional ODEs, their decay rate is an exponential multiplied by a polynomial in time. The novelty of our study lies in the amalgamation of spectral theory and a quantitative non-symmetric hypercontractivity result, as opposed to the usual approach of the entropy method.
\end{abstract}

\nocite{*} %REMOVE!
\maketitle
\section{Introduction}\label{sec:intro}
\subsection{Background}
The study of Fokker-Planck equations (sometimes also called Kolmogorov forward equations) has a long history - going back to the early 20th century. Originally, Fokker and Planck used their equation to describe Brownian motion in a PDE form, rather than its usual SDE representation. \\
In its most general form, the Fokker-Planck equation reads as
\begin{equation}\label{eq:fokker_planck_general}
\partial_t f(t,x) = \sum_{i,j=1}^d \partial_{x_ix_j}\pa{D_{ij}(x) f(t,x)}-\sum_{i=1}^d \partial_{x_i}\pa{A_i(x)f(t,x)},
\end{equation}
with $t>0,x\in\R^d$, and where $D_{ij}(x),A_i(x)$ are real valued functions, with $\D(x)=\pa{D_{ij}(x)}_{i,j=1,\dots,d}$ being a positive semidefinite matrix.\\
The Fokker-Planck equation has many usages in modern mathematics and phys\-ics, with connection to statistical physics, plasma physics, stochastic analysis and mathematical finances. For more information about the equation, we refer the reader to \cite{RiFP89}. Here we will consider a very particular form of \eqref{eq:fokker_planck_general} that allows degeneracies and defectiveness to appear.
\subsection{The Fokker-Planck Equation in our Setting}
In this work we will focus our attention on Fokker-Planck equations of the form:
\begin{equation}\label{eq:fokkerplanck}
\partial_t f(t,x)=Lf(t,x):=\text{div}\pa{\D\nabla f(t,x)+\C xf(t,x)}, \quad\quad t>0, x\in\R^d,
\end{equation}
with appropriate initial conditions, where the matrix $\D$ (the \textit{diffusion} matrix) and $\C$ (the \textit{drift} matrix) are assumed to be constant and real valued.\\ 
In addition to the above,  we will also assume the following: 
\begin{enumerate}\label{hyp1}
\item[(A)] $\D$ is a positive semidefinite matrix with  
$$1\le r:=\text{rank}\pa{\D} \leq d.$$
\item[(B)] All the eigenvalues of $\C$ have positive real part (this is sometimes called \textit{positively stable}).
\item[(C)] There exists no non-trivial $\C^T$-invariant subspace of $\text{Ker}\pa{\D}$ (this is equivalent to \emph{hypoellipticity} of \eqref{eq:fokkerplanck}, cf. \cite{Ho67}).
\end{enumerate}
Each of these conditions has a significant impact on the equation: 
\begin{itemize}
\item Condition (A) allows the possibility that our Fokker-Planck equation is degenerate ($r<d$). 
\item Condition (B) implies that the drift term confines the system. Hence it is crucial for the existence of a non-trivial steady state to the equation, and
\item  Condition (C) tells us that when $\D$ is degenerate, $\C$ compensates for the lack of diffusion in the appropriate direction and ``pushes'' the solution back to where diffusion happens.
\end{itemize}
Equations of the form \eqref{eq:fokkerplanck}, with emphasis on the degenerate structure (and hence $d\ge2$), have been extensively investigated recently (see \cite{AE},\cite{OPPS12}) and were shown to retain much of the structure of their non-degenerate counterpart. When it comes to the question of long time behavior, it has been shown in \cite{AE} that under Conditions (A)-(C) there exists a unique equilibrium state $f_\infty$ to \eqref{eq:fokkerplanck} with a unit mass (it was actually shown that the kernel of $L$ is one dimensional) and that the convergence rate to it can be explicitly estimated by the use of the so called \emph{(relative) entropy functionals}. Based on \cite{AMTU01,BaEmD85}, and denoting by $\R^+:=\br{x>0\;|\;x\in\R}$ and $\R_0^+:=\R^+\cup\br{0}$, we introduce these entropy functionals:
\begin{definition}\label{def-entropy}
We say that a function $\psi$ is a \textit{generating function for an admissible relative entropy} if $\psi \not\equiv 0$, $\psi\in C\pa{\R_0^+}\cap C^4\pa{\R^+}$, $\psi(1)=\psi^\prime(1)=0$, $\psi^{\prime\prime} > 0$ on $\R^+$ and 
\begin{equation}\label{eq:psicondition}
\pa{\psi^{\prime\prime\prime}}^2 \leq \frac{1}{2}\psi^{\prime\prime}\psi^{\prime\prime\prime\prime}.
\end{equation} 
For such a $\psi$, we define the \textit{admissible relative entropy} $e_\psi\pa{\cdot|f_\infty}$ to the Fokker-Planck equation \eqref{eq:fokkerplanck} with a unit mass equilibrium state $f_\infty$, as the functional
\begin{equation}\label{eq:defentropy}
e_\psi\pa{f|f_\infty}:= \int_{\R^d}\psi\pa{\frac{f(x)}{f_\infty(x)}}f_\infty(x)dx,
\end{equation}
for any non-negative $f$ with a unit mass.
\end{definition}
\begin{remark}\label{rem:about_entropies}
It is worth to note a few things about Definition \ref{def-entropy}:
\begin{itemize}
\item As $\psi$ is only defined on $\R^+_0$ the admissible relative entropy can only be used for non-negative functions $f$. This, however, is not a problem for equation \eqref{eq:fokkerplanck} as it propagates non-negativity. 
\item Assumption \eqref{eq:psicondition} is equivalent to the concavity of $(\psi'')^{-1}$ on $\R^+$.
\item  Important examples of generating functions include $\psi_1(y):=y\log y -y+1$ (the Boltzmann entropy) and $\psi_2(y):=\frac{1}{2}(y-1)^2$. \\Note that for $f\in L^2\pa{\R^d,f_\infty^{-1}}$ 
$$e_2(f|f_\infty)=\frac{1}{2}\norm{f-f_\infty}^2_{L^2\pa{\R^d,f_\infty^{-1}}}.$$
This means that up to some multiplicative constant, $e_2$ is the square of the (weighted) $L^2$ norm.
\end{itemize}
\end{remark}
A detailed study of the rate of convergence to equilibrium of the relative entropies for \eqref{eq:fokkerplanck} when $r<d$ was completed recently in \cite{AE}. Denoting by  $L^1_+\pa{\R^d}$ the space of non-negative $L^1$ functions on $\R^d$, the authors have shown the following:
\begin{theorem}\label{thm:Anton_Erb_rate}
Consider the Fokker-Planck equation \eqref{eq:fokkerplanck} with diffusion and drift matrices $\D$ and $\C$ which satisfy Conditions (A)-(C). Let 
\begin{equation}\label{eq:def_of_mu}
\mu:=\min\br{\Re\pa{\lambda}\,|\, \lambda\text{ is an eigenvalue of }\C}.
\end{equation}
Then, for any admissible relative entropy $e_\psi$ and a solution $f(t)$ to \eqref{eq:fokkerplanck} with initial datum $f_0\in L^1_+\pa{\R^d}$, of unit mass and such that $e_\psi(f_0|f_\infty)<\infty$ we have that:
\begin{enumerate}[(i)]
\item If all the eigenvalues from the set
\begin{equation}\label{eq:eigenvalueMu}
\{\lambda \mid \lambda \text{ is an eigenvalue of $\C$ and }\Re(\lambda)=\mu\}
\end{equation} are non-defective
\footnote{An eigenvalue is \textit{defective} if its geometric multiplicity is strictly less than its algebraic multiplicity. We will call the difference between these numbers the \textit{defect} of the eigenvalue.},
 then there exists a fixed geometric constant $c\ge1$, that doesn't depend on $f$, such that
\begin{equation*}%\label{eq:entropydecayAEnodef}
e_\psi(f(t)|f_\infty) \leq c e_\psi(f_0|f_\infty)e^{-2\mu t},\quad t\geq 0.
\end{equation*}
\item If one of the eigenvalues from the set \eqref{eq:eigenvalueMu}
is defective, then for any $\epsilon>0$ there exists a fixed geometric constant $c_{\epsilon}$, that doesn't depend on $f$, such that
\begin{equation}\label{eq:entropydecayAEdef}
e_\psi(f(t)|f_\infty) \leq c_{\epsilon} e_\psi(f_0|f_\infty)e^{-2(\mu-\epsilon) t},\quad t\geq 0.
\end{equation}
\end{enumerate}
\end{theorem} 
The loss of the exponential rate $e^{-2\mu t}$ in part $(ii)$ of the above theorem is to be expected, however it seems that replacing it by $e^{-2\pa{\mu-\epsilon}t}$ is too crude. Indeed, if one considers the much related, finite dimensional, ODE equivalent
\begin{equation}\nonumber
\dot{x}=-\B x
\end{equation}
where the matrix $\B\in\R^{d\times d}$ is positively stable and has, for example, a defect of order $1$ in an eigenvalue with real part equal to $\mu>0$ (defined as in \eqref{eq:def_of_mu}), then one notices immediately that
\begin{equation}\nonumber
{\norm{x(t)}}^2 \leq c\|x_0\|^2\pa{1+t^2}e^{-2\mu t},\quad t\geq 0,
\end{equation} 
i.e.\ the rate of decay is worsened by a multiplication of a polynomial of the order twice the defect of the ``minimal eigenvalue''.\\
The goal of this work is to show that the above is also the case for our Fokker-Planck equation.\\ 
We will mostly focus our attention on the natural family of relative entropies $e_p\pa{\cdot|f_\infty}$, with $1<p\leq 2$, which are generated by
\begin{equation*}%\label{eq:pentropy}
\psi_p(y):=\frac{y^{p}-p(y-1)-1}{p(p-1)}.
\end{equation*}
Notice that $\psi_1$ can be understood to be the limit of the above family as $p$ goes to $1$.\\ 
An important observation about the above family, that we will use later, is the fact that \emph{the generating function for $p=2$, associated to the entropy $e_2$, is actually defined on $\R$ and not only $\R^{+}$}. This is not surprising as we saw the connection between $e_2$ and the $L^2$ norm. This means that we are allowed to use $e_2$ even when we deal with functions without a definite sign.\\
Our main theorem for this paper is the following:
\begin{theorem}\label{thm:main}
Consider the Fokker-Planck equation \eqref{eq:fokkerplanck} with diffusion and drift matrices $\D$ and $\C$ which satisfy Conditions (A)-(C). Let $\mu$ be defined as in \eqref{eq:def_of_mu} and assume that one, or more, of the eigenvalues of $\C$ with real part $\mu$ are defective. Denote by $n>0$ the maximal defect of these eigenvalues. Then, for any $1<p\leq 2$, the solution $f(t)$ to \eqref{eq:fokkerplanck} with unit mass initial datum $f_0\in L^1_+\pa{\R^d}$ and finite $p-$entropy, i.e. $e_p\pa{f_0|f_\infty}<\infty$, satisfies
\begin{equation*}%\label{eq:main}
e_p\pa{f(t)|f_\infty} \leq  \begin{cases}
c_2 e_2\pa{f_0|f_\infty}\pa{1+t^{2n}}e^{-2\mu t}, & p=2, \\
c_p \pa{p(p-1)e_p(f_0|f_\infty)+1}^{\frac{2}{p}}\pa{1+t^{2n}}e^{-2\mu t}, & 1<p<2,
\end{cases}
\end{equation*}
for $t\ge0$, where $c_p>0$ is a fixed geometric constant, that doesn't depend on $f_0$, and $f_\infty$ is the unique equilibrium with unit mass.
\end{theorem}
The main idea, and novelty, of this work is in combining elements from Spectral Theory and the study of our $p-$entropies. We will give a detailed study of the geometry of the operator $L$ in the $L^2\pa{\R^d,f_\infty^{-1}}$ space and deduce, from its spectral properties, the result for $e_2$. Since the other entropies, $e_p$ for $1<p<2$, lack the underlying geometry of the $L^2$ space that $e_2$ enjoys, we will require additional tools: We will show a quantitative result of \emph{hypercontractivity for non-symmetric Fokker-Planck operators} that will assure us that after a certain, \emph{explicit} time, any solution to our equation with finite $p-$entropy will belong to $L^2\pa{\R^d,f_\infty^{-1}}$. This, together with the dominance of $e_2$ over $e_p$ for functions in $L^2\pa{\R^d,f_\infty^{-1}}$ will allow us to ``push'' the spectral geometry of $L$ to solutions with initial datum that only has finite $p-$entropy.\\
\amit{We have recently become aware that the long time behaviour of Theorem \ref{thm:main} has been shown in a preprint by Monmarch\'e, \cite{MoH15Arx}. However, the method he uses to show this result is a generalised entropy method (more on which can be found in \S\ref{sec:fisher}), while we have taken a completely different approach to the matter.\\}
The structure of the work is as follows: In \S\ref{sec:fokkerplanck} we will recall known facts about the Fokker-Planck equation (degenerate or not).  \S\ref{sec:spectral} will see the spectral investigation of $L$ and the proof of Theorem \ref{thm:main} for $p=2$. In \S\ref{sec:hyper}  we will show our non-symmetric hypercontractivity result and conclude the proof of our Theorem \ref{thm:main}. Lastly, in \S\ref{sec:fisher} we will recall another important tool in the study of Fokker-Planck equations - the Fisher information - and show that Theorem \ref{thm:main} can also be formulated for it, due to the hypoelliptic regularisation of the equation. 
\section{The Fokker-Planck Equation}\label{sec:fokkerplanck}
This section is mainly based on recent work of Arnold and Erb (see \cite{AE}). We will provide here, mostly without proof, known facts about degenerate (and non-degenerate) Fokker-Planck equations of the form \eqref{eq:fokkerplanck}.
\begin{theorem}\label{thm:propertiesoffokkerplanck}
Consider the Fokker-Planck equation \eqref{eq:fokkerplanck}, with diffusion and drift matrices $\D$ and $\C$ that satisfy Conditions (A)-(C), and an initial datum $f_0\in L^1_+\pa{\R^d}$. Then
\begin{enumerate}[(i)]
\item There exists a unique classical solution $f\in C^\infty \pa{\R^+\times \R^d}$ to the equation. Moreover, if $f_0\not=0$ it is strictly positive for all $t>0$.
\item For the above solution $\int_{\R^d}f(t,x)dx=\int_{\R^d}f_0(x)dx$.
\item If in addition $f_0\in L^{p}\pa{\R^d}$ for some $1< p\leq \infty$, then $f\in C\pa{[0,\infty),L^p\pa{\R^d}}$. 
\end{enumerate}
\end{theorem}
\begin{theorem}\label{thm:equilibrium}
Assume that the diffusion and drift matrices, $\D$ and $\C$, satisfy Conditions (A)-(C). Then, there exists a unique stationary state $f_\infty\in L^1\pa{\R^d}$ to \eqref{eq:fokkerplanck} satisfying $\int_{\R^d} f_\infty(x) dx = 1$. Moreover, $f_\infty$ is of the form:
\begin{equation}\label{eq:equilibrium}
f_\infty(x)= c_{\K} e^{-\frac{1}{2}x^T \K^{-1} x},
\end{equation}
where the covariance matrix $\K\in\R^{d\times d}$ is the unique, symmetric and positive definite solution to the continuous Lyapunov equation
\begin{equation}\nonumber
2\D=\C \K+\K \C^T,
\end{equation} 
and where $c_{\K}>0$ is the appropriate normalization constant. In addition, for any $f_0\in L^1_+\pa{\R^d}$ with unit mass, the solution to the Fokker-Planck equation \eqref{eq:fokkerplanck} with initial datum $f_0$ converges to $f_\infty$ in relative entropy (as referred to in Theorem \ref{thm:Anton_Erb_rate}).
\end{theorem}
\begin{remark}\label{rem:equilibrium}
In the case where $f_0\in L^1_+\pa{\R^d}$ is not of unit mass, it is immediate to deduce that the solution to the Fokker-Planck equation with initial datum $f_0$ converges to $\pa{\int_{\R^d}f_0(x)dx}f_\infty(x)$.
\end{remark}
\begin{corollary}\label{cor:simple_form_for_L}
The Fokker-Planck operator $L$ can be rewritten as 
\begin{equation}\label{eq:FPrecast-gen}
  Lf=\text{div}\pa{f_\infty(x)\C\K \nabla \pa{\frac{f(t,x)}{f_\infty(x)}}}
\end{equation}
(cf.\ Theorem 3.5 in \cite{AE}).
\end{corollary}
A surprising, and useful, property of \eqref{eq:fokkerplanck} is that the diffusion and drift matrices associated to it can always be simplified by using a change of variables. The following can be found in \cite{AAS}:
\begin{theorem}\label{thm:simpliedDC}
Assume that the diffusion and drift matrices satisfy Conditions (A)-(C). Then, there exists a \textit{linear} change of variable that transforms \eqref{eq:fokkerplanck} to itself with new diffusion and drift matrices $\bf D$ and $\bf C$ such that 
\begin{equation}\label{eq:diagD}
\D=\operatorname{diag}\br{d_1,d_2,\dots,d_r,0,\dots,0}
\end{equation} 
with $d_j>0$, $j=1,\ldots,r$ and $\C_s:=\frac{\C+\C^T}{2}=\D$. In these new variables the equilibrium $f_\infty$ is just the standard Gaussian with $\K=\II$.
\end{theorem} 
The above matrix normalisation has additional impact on the calculation of the adjoint operator: 
\begin{corollary}
Let $\C_s=\D$. Then:
\begin{enumerate}
\item[(i)] 
\begin{equation*}
  \big(L_{\D,\C}\big)^* = L_{\D,\C^T}\,,
\end{equation*}
where $L^*$ denotes the (formal) adjoint of $L$, considered w.r.t.\ $L^2\pa{\R^d,f_\infty^{-1}}$. The domain of $L$ will be discussed in \S\ref{sec:spectral}.
\item[(ii)]
The kernels of $L$ and $L^*$ are both spanned by $\exp(-\frac{|x|^2}{2})$. This is not true in general, i.e.\ for a Fokker-Planck operator $L$ without the matrix normalisation assumption.
%i.e.\ without the under the normalizing coordinate transformation of Theorem \ref{thm:simpliedDC}). 
\end{enumerate}
\end{corollary}
\begin{proof}
(i) Under the normalising coordinate transformation of Theorem \ref{thm:simpliedDC} we see from \eqref{eq:FPrecast-gen} that % for functions in $D(L)$
\begin{equation}
\begin{split}\label{L-star}
\int_{\R^d} f(x) L_{\D,\C}&g(x) f_\infty^{-1}(x)dx= -\int_{\R^d} f_\infty(x) \nabla \pa{\frac{f(x)}{f_\infty(x)}}^T\C \nabla \pa{\frac{g(x)}{f_\infty(x)}}dx \\
=& \int_{\R^d} \text{div}\pa{f_\infty(x) \C^ T\nabla\pa{\frac{f(x)}{f_\infty(x)}}} g(x)f_\infty^{-1}(x)dx.
\end{split}
\end{equation}
(ii) follows from \eqref{eq:equilibrium} and $\K=\II$.
\end{proof}

{}From this point onwards we will always assume that Conditions (A)-(C) hold, and that we are in the coordinate system where $\D$ is of form \eqref{eq:diagD} and equals $\C_s$.\\

%%%%%%%%%%%%%%%%%%%%%%%%%%%%%%%%%%%%%%%%%%%%%%%%%%%%%%%%%%%%%%%%%%%%%%%%%%%%%%%%%%%%%%%%%%%%%%%%%%%%%%%%%%%%%%%%%

\section{The Spectral Study of $L$}\label{sec:spectral}
The main goal of this section is to explore the spectral properties of the Fokker-Planck operator $L$ in $L^2\pa{\R^d,f_\infty^{-1}}$, and to see how one can use them to understand rates of convergence to equilibrium for $e_2$. The crucial idea we will implement here is that, since $L^2\pa{\R^d,f^{-1}_{\infty}}$ decomposes into orthogonal eigenspaces of $L$ with eigenvalues that get increasingly farther to the left of the imaginary axis, one can deduce \emph{improved convergence rates on ``higher eigenspaces''}. \\
The first step in achieving the above is to recall the following result from \cite{AE}, where we use the notation $\N_0:=\N\cup\{0\}$:
\begin{theorem}\label{thm:spectraldecomposition}
Denote by 
\begin{equation*}%\label{eq:defVm}
V_m:=\span\br{\partial_{x_1}^{\alpha_1}\dots \partial_{x_d}^{\alpha_d}f_\infty(x)\ \Big|\ \alpha_1,\dots,\alpha_d\in \N_0, \sum_{i=1}^d \alpha_i=m}.
\end{equation*}
Then, $\br{V_m}_{m\in \N_0}$ are mutually orthogonal in $L^2\pa{\R^d, f_\infty ^{-1}}$, 
\begin{equation}\nonumber
L^2\pa{\R^d,f_\infty^{-1}}= \bigoplus_{m\in\N_0}V_m,
\end{equation}
and $V_m$ are invariant under $L$ and its adjoint (and thus under the flow of \eqref{eq:fokkerplanck}).\\
Moreover, the spectrum of $L$ satisfies
\begin{equation}\nonumber
\begin{gathered}
\sigma\pa{L}=\bigcup_{m\in\N_0}\sigma\pa{L\vert_{V_m}},\\
\sigma\pa{L\vert_{V_m}}=\br{-\sum_{i=1}^d \alpha_i \lambda_i \ \Big|\ \alpha_1 ,\dots,\alpha_d\in\N_0, \sum_{i=1}^d \alpha_i=m},
\end{gathered}
\end{equation}
where $\br{\lambda_j}_{j=1,\dots,d}$ are the eigenvalues (with possible multiplicity) of the matrix $\C$. The eigenfunctions of $L$ (or eigenfunctions and generalized eigenfunctions in the case $\C$ is defective) form a basis to $L^2\pa{\R^d,f_\infty^{-1}}$.
\end{theorem}
Let us note that this orthogonal decomposition is non-trivial since $L$ is in general non-symmetric. The above theorem quantifies our previous statement about ``higher eigenspaces'': the minimal distance between the eigenvalues of $L$ restricted to the ``higher'' $L$-invariant eigenspace $V_m$ and the imaginary axis is $m\mu$. Thus, the decay we expect to find for initial datum from $V_m$ is of order $e^{-2m\mu t}$ (in the quadratic entropy, e.g.). However, as the function we will use in our entropies are not necessarily contained in only finitely many $V_m$, we might need to pay a price in the rate of convergence.\\ 
This intuition is indeed true. Denoting by
\begin{equation}\label{def:Hk}
H_k := \bigoplus _{m\geq k} V_m
\end{equation}
for any $k\geq 0$, we have the following:
\begin{theorem}\label{thm:e2rateofdecay}
Let \amit{$ f_k\in H_{k}$ for some $k\geq 1$} and let $f(t)$ be the solution to \eqref{eq:fokkerplanck} with initial data \amit{$f_0=f_\infty +  f_k$}. Then for any $0<\epsilon<\mu$ there exists a geometric constant $c_{k,\epsilon}\ge1$ that depends only on $k$ and $\epsilon$ such that 
\amit{\begin{equation} \label{eq:e2rateofdecay}
 e_2\pa{f(t)|f_\infty} \leq c_{k,\epsilon}e_2(f_0|f_\infty) e^{-2(k\mu-\epsilon)t}\,,\quad t\ge0\,.
 \end{equation}}
\end{theorem}
\begin{remark}\label{rem:comparisonDecay}
The loss of an $\epsilon$ in the decay rate of \eqref{eq:e2rateofdecay} -- compared to the decay rate solely on $V_k$ -- can have two causes: 
\begin{enumerate}
\item For drift matrices $\C$ with a defective eigenvalue with real part $\mu$, the larger decay rate $2k\mu$ would not hold in general. This is illustrated in \eqref{eq:entropydecayAEdef}, which provides the best possible \textit{purely exponential} decay result, as proven in \cite{AE}. 
\item For \emph{non-defective matrices} $\C$, the improved decay rate $2k\mu$ actually holds, but our method of proof, that uses the Gearhart-Pr\"uss Theorem, cannot yield this result.
% In the following proof we shall sketch an improved proof strategy for $k=1$. It is based on splitting the dominant contribution in the lowest relevant eigenspace from the rest, i.e. faster decaying contributions in higher eigenspaces. The essence is here to use the orthogonality of the eigenspaces $V_m$. This makes our procedure feasible for the quadratic entropy, but not for general entropies.
The decay estimate \eqref{eq:e2rateofdecay} will be improved in Theorem \ref{thm:rate_of_decay_e2}: There, the $\epsilon$-reduction drops out in the non-defective case.
\end{enumerate}
\end{remark}
\begin{remark}\label{rem:e2decay_no_positivity}
As we insinuated in the introduction to our work, an important observation to make here is that the initial data, $f_0$, \emph{doesn't have to be non-negative} (and in many cases, is not). While this implies that $f(t)$ might also be non-negative, this poses no problems as $e_2$ is the squared (weighted) $L^2$ norm (up to a constant). Theorem \ref{thm:e2rateofdecay} \emph{would not work in general} for $e_p$ as the non-negativity of $f(t)$ is crucial there (in other words, $f_0$ would not be admissible).
\end{remark}
The main tool to prove Theorem \ref{thm:e2rateofdecay} is the Gearhart--Pr\"uss Theorem (see for instance Th.\ 1.11 Chap.\ V in \cite{EN00}). In order to be able to do that, we will need more information about the dissipativity of $L$ and its resolvents with respect to $H_k$.
\begin{lemma}\label{lem:dissipative}
Let $V_m$ be as defined in Theorem \ref{thm:spectraldecomposition}. Consider the operator $L$ with the domain $D(L)=\span\br{V_m,\, m\in\N_0}$. Then $L$ is dissipative, and as such closable. Moreover, its closure, $\overline{L}$, generates a contraction semigroup on $L^2\pa{\R^d,f_\infty^{-1}}$.
 \end{lemma}
\begin{proof}
Given $f\in D(L)$, and denoting $g:=\frac{f}{f_\infty}$, we notice that \eqref{eq:FPrecast-gen} with $\K=\II$ implies that
$$\pa{Lf,f}_{L^2\pa{\R^d,f_\infty^{-1}}} = \int_{\R^d} \text{div}\pa{f_\infty (x)\C \nabla g(x)}g(x)dx = -\int_{\R^d} \nabla g(x)^T \C \nabla g(x) f_\infty(x)dx$$
$$=-\int_{\R^d} \nabla g(x)^T \D \nabla g(x) f_\infty(x)dx \leq 0,$$
where we have used the fact that $\C_s=\D$. Thus, $L$ is dissipative.\\
To show the second statement we use the Lumer-Phillips Theorem (see for instance Th. 3.15 Chap. II in \cite{EN00}). Since $L^2\pa{\R^d,f_\infty^{-1}}=\bigoplus_{m\in\N_0} V_m$ it will be enough to show that for $\lambda>0$ we have that $V_m\subset \text{Range}\pa{\lambda I-L}$ for any $m$. As $V_m \subset D(L)$, is finite dimensional, and is invariant under $L$ (Theorem \ref{thm:spectraldecomposition} again) we can consider the linear bounded operator $L\vert_{V_m}:V_m \rightarrow V_m$. Since we have shown that $L$ is dissipative, we can conclude that the eigenvalues of $L\vert_{V_m}$ have non-positive real parts, implying that $\pa{\lambda I- L} \vert_{V_m}$ is invertible. This in turn implies that
$$V_m=\text{Range}\pa{\pa{\lambda I-L}\vert_{V_m}}\subset \text{Range}\pa{\lambda I-L},$$
completing the proof.
\end{proof}
To study the resolvents of $L$ we will need to use some information about its ``dual'': the Ornstein-Uhlenbeck operator.\\
For a given symmetric positive semidefinite matrix ${\bf Q}=(q_{ij})$ and a real, negatively stable matrix $\B=(b_{ij})$ on $\R^d$ we consider the Ornstein-Uhlen\-beck operator
\begin{equation}\label{eq:OU_operator}
P_{{\bf{Q}},\B}:=\frac{1}{2}\sum_{i,j}q_{ij}\partial_{x_ix_j}^2+\sum_{i,j}b_{ij}x_j \partial_{x_i}=\frac{1}{2}\tr\pa{{\bf{Q}} \nabla_x^2}+\pa{\B x,\nabla_x},\quad x\in\R^d.
\end{equation}
Similarly to our conditions on the diffusion and drift matrices, we will only be interested in Ornstein-Uhlen\-beck operators that are \emph{hypoelliptic}. In the above setting, this corresponds to the condition
$$\rank \left[{\bf{Q}}^{\frac{1}{2}},\B{\bf{Q}}^{\frac{1}{2}},\dots,\B^{d-1} {\bf{Q}}^{\frac{1}{2}} \right]=d.$$
The hypoellipticity condition guarantees the existence of an invariant measure, $d\mu$, to the process. This measure has a density w.r.t.\ the Lebesgue measure, which is given by 
$$\frac{d\mu}{dx}(x)= c_\M e^{-\frac12 x^T\M^{-1}x}\,,\quad\mbox{with}\quad
	{\bf \M}:=\int_0^\infty e^{\B s}\Q e^{\B^T s}\,ds$$
where $c_\M>0$ is a normalization constant. It is well known that the above definition of ${\bf \M}$ is equivalent to finding the unique solution to the continuous Lyapunov equation 
\begin{equation}\label{Lyapunov}
  \Q=-\B\M-\M\B^T\,.
\end{equation}
(See for instance Theorem 2.2 in \cite{SnZaN70}, \S2.2 of \cite{HoJoT91}.)\\
Hypoelliptic Ornstein-Uhlen\-beck operators have been studied for many years, and more recently in \cite{OPPS15} the authors considered them under the additional possibility of degeneracy in their diffusion matrix ${\bf{Q}}$. In \cite{OPPS15}, the authors described the domain of the closed operator $P_{{\bf{Q}},\B}$, and have found the following resolvent estimation:
\begin{theorem}\label{thm:OPPSthm}
Consider the hypoelliptic Ornstein-Uhlen\-beck operator $P_{{\bf{Q}},\B}$, as in \eqref{eq:OU_operator}, and its invariant measure $d\mu(x)$. Then there exist some positive constants $c,C>0$ such that for any $z\in \Gamma_{\kappa}$, with 
\begin{equation}\label{eq:OPPSset}
\Gamma_{\kappa}:=\br{z\in\mathbb{C} \,\Bigg |\, \Re z \leq \textstyle\frac{1}{2}\pa{1-\tr(\B)}, \abs{\Re z-\pa{1-\textstyle\frac{1}{2}\tr(\B)}} \leq  c\abs{z-\pa{1-\textstyle\frac{1}{2}\tr(\B)}}^{\frac{1}{2\kappa+1}} }
\end{equation}
and where $\kappa$ is the smallest integer $0\leq \kappa\leq d-1$ such that
\begin{equation}\label{rank-cond}
\rank \left[{\bf{Q}}^{\frac{1}{2}},\B{\bf{Q}}^{\frac{1}{2}},\dots,\B^{\kappa} {\bf{Q}}^{\frac{1}{2}} \right]=d\ ,
\end{equation}
one has that
\begin{equation*} %\label{eq:OPPSresolvantestimation}
\norm{\pa{P_{{\bf{Q}},\B}-zI}^{-1}}_{B\pa{L^2\pa{\R^d,d\mu}}} \leq C\abs{z-\pa{1-\frac{1}{2}\tr(\B)}}^{-\frac{1}{2\kappa+1}}.
\end{equation*}
\end{theorem}
We illustrate the spectrum of $P_{\Q,\B}$ and the domain $\Gamma_{\kappa}$ in Figure \ref{fig:gamma}.
\begin{figure}[ht]
	\centering
  \includegraphics[scale=0.8]{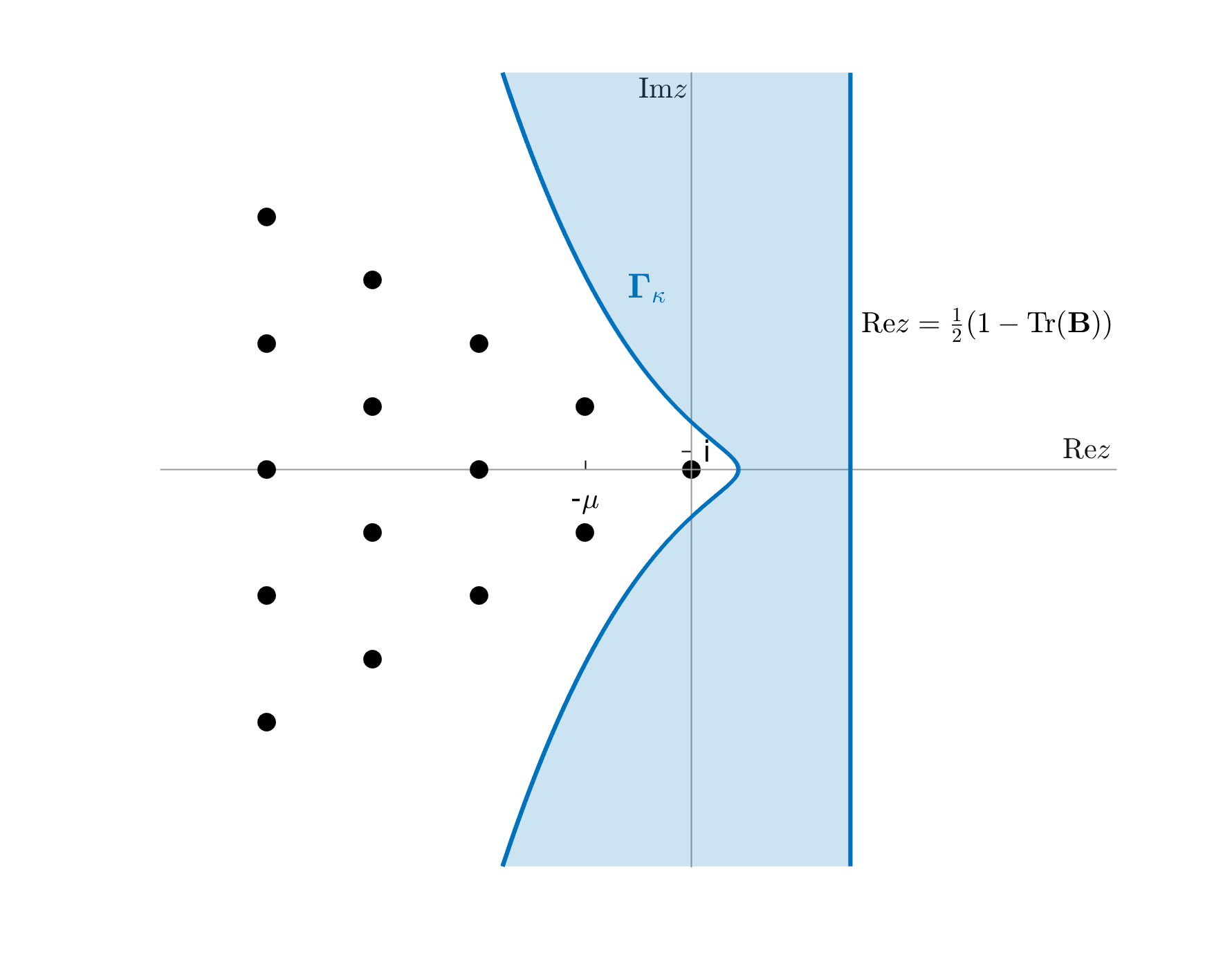}
	\caption{The black dots represent $\sigma(P_{\Q,\B})$ with the eigenvalues of the $2\times 2$ matrix $\B$ given as $\lambda_{1,2}=-1\pm \frac{7}{2}i$. The shaded area represents the set $\Gamma_{\kappa}$ of Theorem \ref{thm:OPPSthm} with $\kappa=1$.}
	\label{fig:gamma}
\end{figure}\\
In order to use the above theorem for our operator, $L$, we show the connection between it and $P$ in the following lemma:
\begin{lemma}\label{lem:equiv_L_and_P}
Assume that the associated diffusion and drift matrices for $L$, defined on $L^2\pa{\R^d,f_\infty^{-1}}$, and $P_{\Q,\B}$, defined on $L^2\pa{\R^d,d\mu(x)}$, satisfy
$$\Q=2\D,\,\B=-\C.$$ 
Then $d\mu(x)=f_\infty(x)dx$ is the invariant measure for $P=P_{\Q,\B}$ and its adjoint, and (up to the natural transformation $\frac{Lf}{f_\infty}=P^*(\frac{f}{f_\infty})$) we have $L=P^*$.
\end{lemma}
\begin{proof}
We start by recalling that we assume that $\D=\C_s$. Since \eqref{Lyapunov} can be rewritten as
$$2\D=\C\M+\M\C^T$$
for our choice of ${\Q} $ and $\B$, we conclude that $\M=\II$ for $P_{2\D,-\C}$ and that
 $\big(P_{2\D,-\C}\big)^*=P_{2\D,-\C^T}$ (the last equality can be shown in a similar way to \eqref{L-star}). Thus, the invariant measure corresponding to both these operators is $f_\infty(x) dx$.\\
Let $f\in D(L)\subset L^2\pa{\R^d,f_\infty^{-1}}$  and define $g_f:=\frac{f}{f_\infty} %$. We have that $g_f
\in L^2\pa{\R^d,f_\infty}$.
Then
\begin{equation}\label{eq:connectionPL}
\begin{gathered}
\frac{L_{\D,\C} f(x)}{f_\infty(x)}=\frac{\text{div}\pa{f_\infty(x)\C \nabla g_f(x)}}{f_\infty(x)} = \text{div}\pa{\C \nabla g_f(x)} + \frac{\nabla f_\infty(x)^T \C \nabla g_f(x)}{f_\infty(x)}\\
=\text{div}\pa{\D \nabla g_f(x)} - x^T \C \nabla g_f(x)=P_{2\D,-\C^T}g_f(x)=\big(P_{2\D,-\C}\big)^*g_f(x)\,,
\end{gathered}
\end{equation}
where the adjoint is considered w.r.t.\ $L^2\pa{\R^d,f_\infty}$. In particular, if $f(t,\cdot)\in L^2\pa{\R^d,f_\infty^{-1}}$ solves \eqref{eq:fokkerplanck} then $g_f(t,\cdot)$ satisfies the adjoint equation $\partial_t g_f =\big(P_{2\D,-\C}\big)^*g_f$. 
\end{proof}

With this at hand we can recast, and improve, Theorem \ref{thm:OPPSthm} for the operator $L$ and its closure.
\begin{prop}\label{prop:OPPSthmimproved}
Let any $k\in\N_0$ be fixed.
Consider the set $\Gamma_{\kappa}$, defined by \eqref{eq:OPPSset}, associated to $\Q=2\D,\,\B=-\C^T$ (Condition (C) guarantees the existence of such $\kappa$). Then we have that, for any $z\in \Gamma_{\kappa}$, the operator $\pa{L-zI}\vert_{H_k}:H_k\rightarrow H_k$ 
is well defined, closable, and its closure is invertible with
\begin{equation}\label{eq:OPPSimprovedresolvant}
\Norm{\pa{\pa{\overline{L}-zI}\vert_{H_k}}^{-1}}_{B\pa{H_k}} \leq C\abs{z-\pa{1+\frac{1}{2}\tr(\C)}}^{-\frac{1}{2\kappa+1}},
\end{equation} 
where $C>0$ is the same constant as in Theorem \ref{thm:OPPSthm}.
\end{prop}
\begin{proof}
We consider the case $k=0$ first. Due to Theorem \ref{thm:OPPSthm} we know that for any $z\in\Gamma_{\kappa}$, $P_{2\D,-\C^T}-zI$ is invertible on $L^2\pa{\R^d,f_\infty}$. Hence, for any $f\in L^2\pa{\R^d,f_\infty^{-1}}$ there exists a unique $\ell_f\in L^2\pa{\R^d,f_\infty}$ such that
$$\pa{P_{2\D,-\C^T}-zI}\ell_f(x)=\frac{f(x)}{f_\infty(x)},$$
which can also be written differently due to \eqref{eq:connectionPL}, as 
$$\pa{\overline{L}-zI} \pa{f_\infty(x) \ell_f(x)}=f(x).$$
This implies that $\overline{L}-zI$ is bijective on its appropriate space.\\ 
Next we notice that, with the notations from Lemma \ref{lem:equiv_L_and_P}
$$\sup_{\norm{f}=1}\norm{\pa{\overline{L}-zI}^{-1}f}_{L^2\pa{\R^d,f_\infty^{-1}}}=\sup_{\norm{f}=1}\norm{f_\infty \ell_f}_{L^2\pa{\R^d,f_\infty^{-1}}}$$
$$=\sup_{\norm{f}=1}\norm{\ell_f}_{L^2\pa{\R^d,f_\infty}}=\sup_{\norm{g_f}=1}\norm{\pa{P_{2\D,-\C^T}-zI}^{-1}g_f}_{L^2\pa{\R^d,f_\infty}}\,,$$
from which we conclude that
$$\norm{\pa{\overline{L}-zI}^{-1}}_{B\pa{L^2\pa{\R^d,f_\infty^{-1}}}}=\norm{\pa{P_{2\D,-\C^T}-zI}^{-1}}_{B\pa{L^2\pa{\R^d,f_\infty}}}\,,$$
completing the proof for this case.\\
We now turn our attention to the restrictions $\pa{L-zI}\vert_{H_k}$ with $k\ge1$ and domain 
$$D_k:=\span\br{V_m, m\geq k}=D(L)\cap H_k.$$ 
Since $L|_{V_m}:\,V_m\to V_m$ $\forall m\in\N_0$ we have that $\pa{L-zI}\vert_{H_k}:\,D_k\to H_k$. Moreover, the dissipativity of $L$ on $D(L)$ assures us that $L$ is dissipative, and as such closable, on the Hilbert space $H_k$. Thus $(L-zI)\vert_{H_k}$ is closable too and 
$$\overline{\pa{L-zI}\vert_{H_k}}=\pa{\overline{L}-zI}\vert_{H_k}.$$
Additionally, since the only part of $L^2\pa{\R^d,f_\infty^{-1}}$ that is not in $H_k$ is a finite dimensional subspace of $D(L)$, we can conclude that
$$D\left((\overline{L}-zI)\vert_{H_k}\right)=D(\overline{L})\cap H_k.$$
Given $z$ in the resolvent set of $\overline{L}$ we know that $\overline{L}-zI\vert_{V_m}:V_m\rightarrow V_m$ is invertible for any $m$ and as such
$$(\overline{L}-zI)\vert_{V_m}\pa{V_m}=V_m.$$
Thus,
$$V_m\subset\range\pa{(\overline{L}-zI)\vert_{H_k}},\qquad \forall m \ge k.$$ 
We conclude that $(\overline{L}-zI) \vert_{H_k}$ is injective with a dense range in $H_k$ for any $z\in\Gamma_{\kappa}$, and hence invertible on its range. The validity of  \eqref{eq:OPPSimprovedresolvant} for $k=0$ allows us to extend our inverse to $H_k$ with the same \emph{uniform bound} as is given in \eqref{eq:OPPSimprovedresolvant}. The general case is now proved.
\begin{comment}
The fact that $L$, with domain $D_k$, is well defined on $H_k$ is immediate if one considers the invariance under $L$ of all $V_m$, and the fact that $\oplus_{m=0}^{k-1} V_m$ is finite dimensional. Moreover, $L$ is dissipative. As the semigroup its closure generates over $H_k$ must coincides with the one generated by the closure of $L$ over all of $H$, we can view the above as $L\vert_{H_k}$.\\
Consider $z\in \Gamma_{\kappa}$. Since for any $0\leq j <l$ the subspace $\oplus_{m=j}^{l}V_m$ is finite dimensional, contained in $D(L)$ and invariant under $L$, we have that $\pa{L-zI}\vert_{\oplus_{m=j}^{l}V_m}$ is invertible. Moreover, due to the orthogonality of $\br{V_m}_{m\in\N_0}$
$$\norm{\pa{L-zI}\vert_{\oplus _{m=j}^l V_m}^{-1}}_{B\pa{\oplus _{m=k}^n V_m}}\leq \norm{\pa{L-zI}^{-1}}_{B\pa{L^2\pa{\R^d,f_\infty^{-1}}}}$$
The above states that $L-zI$ has a bounded inverse on $D_k$, which is dense in $H_k$. Thus, $\pa{\pa{L-zI}\vert_{H_k}}^{-1}$ is defined and satisfies \eqref{eq:OPPSimprovedresolvant}.
\end{comment}
\end{proof}
{}From this point onward, we will assume that we are dealing with the closed operator $\overline{L}$ and with its appropriate domain (that includes $\bigcup_{m\in\N_0}V_m$) when we consider our equation. We will also write $L$ instead of $\overline{L}$ in what is to follow.\\
Lemma \ref{lem:dissipative} and Proposition \ref{prop:OPPSthmimproved} are all the tools we need to estimate the uniform exponential stability of our evolution semigroup on each $H_k$, an estimation that is crucial to show Theorem \ref{thm:e2rateofdecay}.
\begin{prop}\label{prop:exp_stability}
Consider the Fokker-Planck operator $L$, defined on $L^{2}\pa{\R^d,f_\infty^{-1}}$, and the spaces $\br{H_k}_{k\geq 1}$ defined in \eqref{def:Hk}. Then, for any $0<\epsilon<\mu$, the semigroup generated by the operator $L+\pa{k\mu-\epsilon}I\vert_{H_k}$, with domain $D(L)\cap H_k$, is uniformly exponentially stable. I.e., there exists some geometric constant $C_{k,\epsilon}>0$ such that 
\begin{equation}\label{eq:GPresult}
\norm{e^{Lt}}_{B\pa{H_k}} \leq C_{k,\epsilon}e^{-\pa{k\mu-\epsilon}t}\,,\quad t\ge0\,,
\end{equation}
\end{prop}

\begin{proof}
We will show that
$$M_{k,\epsilon}:=\sup_{\Re z>0}\left\|\bigg(\pa{L+[k\mu-\epsilon]I}-zI\bigg)^{-1}\right\|_{B(H_k)} < \infty\,,$$
and conclude the result from the fact that $L$ generates a contraction semigroup according to Lemma \ref{lem:dissipative} and the Gearhart-Pr\"uss Theorem.\\
The study of upper bounds for the resolvents of $L+[k\mu-\epsilon]I$  in the right-hand complex plane relies on subdividing this domain into several pieces. This is illustrated in Figure \ref{fig:compact}, which we will refer to during the proof to help visualise this division.\\
%\ref{lem:dissipative}}}
Since $L$ generates a contraction semigroup, for any $\epsilon>0$, $L-\epsilon I$ generates a semigroup that is uniformly exponentially stable on $L^2(\R^d,f_\infty^{-1})$. 
The Gearhart-Pr\"{u}ss Theorem applied to $L-\epsilon I$ implies that
\[
\widetilde{M}_{k,\epsilon}:=\sup_{\Re z>0} \left\|\left(L-(\epsilon+z)I\right)^{-1}\right\|_{B(H_k)} \leq \sup_{ \Re z>0}\left\|\left(L-(\epsilon+z)I\right)^{-1} \right\|_{{B}(L^2(\R^d,f_\infty^{-1}))}<\infty,
\]
where we removed the subscript $H_k$ from the operator on the left-hand side to simplify notations.\\
Since
$$L-\pa{\epsilon+z}I=L+[k\mu-\epsilon]I-\pa{z+k\mu}I,$$ 
we see that 
$$\widetilde{M}_{k,\epsilon}=\sup_{\Re z_1>0} \left\| \bigg(\pa{L+[k\mu-\epsilon]I}-\pa{z_1+k\mu}I\bigg)^{-1}\right\|_{B(H_k)}$$
$$=\sup_{\Re z>k\mu} \left\| \bigg(\pa{L+[k\mu-\epsilon]I}-zI\bigg)^{-1}\right\|_{B(H_k)}$$
(this term corresponds to the right-hand side of the dashed line in Figure \ref{fig:compact}).\\
{}From the above we conclude that
$$M_{k,\epsilon}= \max\pa{\widetilde{M}_{k,\epsilon} \,,\sup_{0<\Re z\leq k\mu} \left\|\pa{L-[z-k\mu+\epsilon]I}^{-1}\right\|_{B(H_k)}}.$$
which implies that we only need to show that the second term in the parenthesis is finite (this term corresponds to the area between the dashed line and the imaginary axis in Figure \ref{fig:compact}). \\
\begin{figure}[ht]
	\centering
  \includegraphics[scale=0.8]{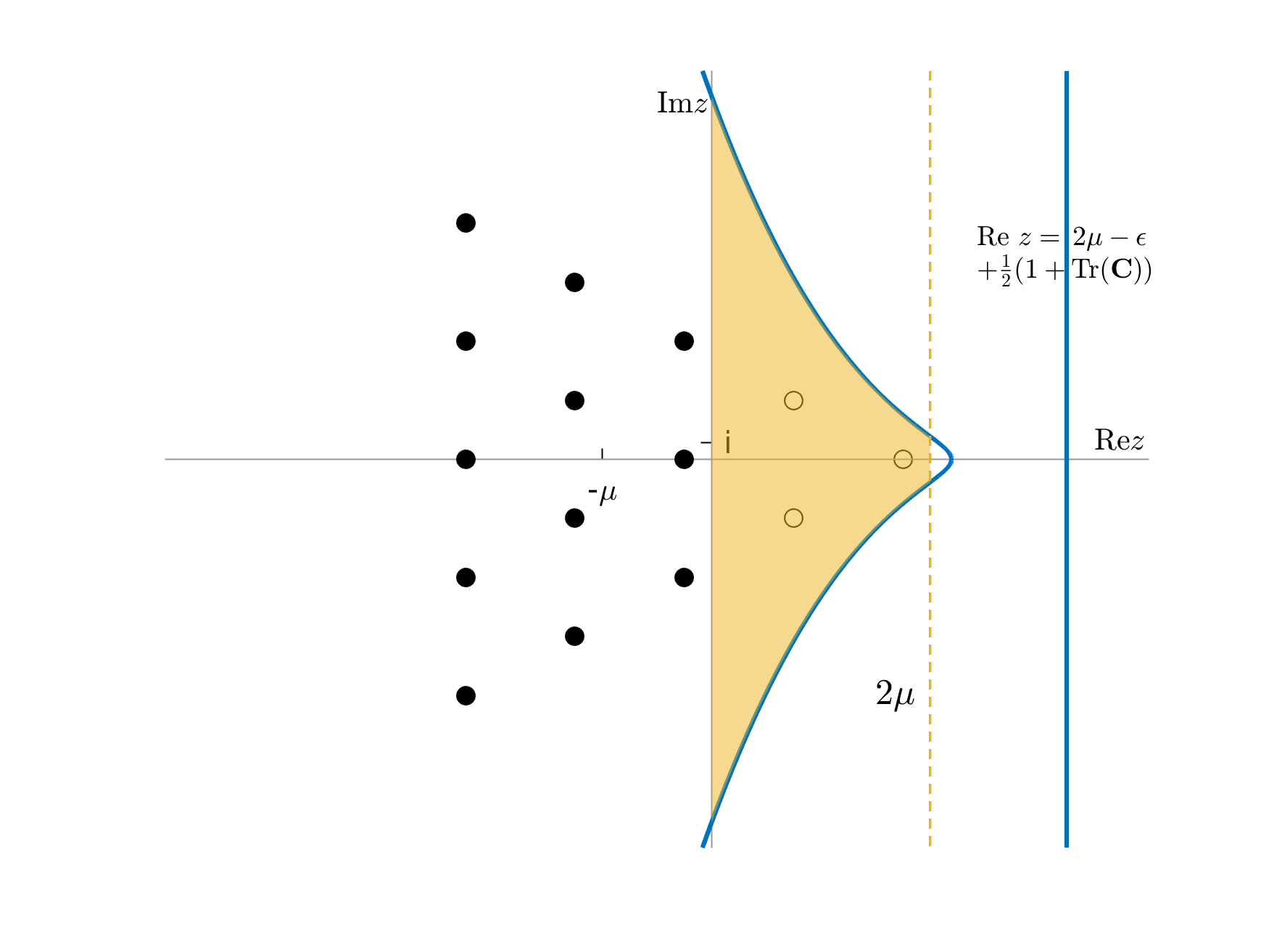}
	\caption{ choosing $k=2$, the solid dots represent $\sigma((L+[2\mu-\epsilon]I)|_{H_{2}})$ where the eigenvalues of the $2 \times 2$ matrix  $\C$ are given by $\lambda_{1,2}=1\pm \frac{7}{2}i$. The empty dots are the eigenvalues of the operator $L+[2\mu-\epsilon]I$ that disappear due to the restriction to $H_2$, and the shaded area represents the compact set $\{z\in\mathbb{C} \mid 0\leq \Re z \leq 2\mu\}\cap \{z\not\in \Gamma_{\kappa}+2\mu-\epsilon\}$ where $\kappa=1$.}
	\label{fig:compact}
\end{figure}
Using Proposition \ref{prop:OPPSthmimproved} we conclude that
\[
\sup_{z-k\mu+\epsilon\in\Gamma_{\kappa}} \left\| \left(L-\left[z-k\mu+\epsilon \right]I\right)^{-1}\right\|<\infty
\] 
(rep\-re\-sent\-ed in Figure \ref{fig:compact} by the domain between the two solid blue curves). We conclude that $M_{k,\epsilon}<\infty$ if and only if
$$\sup_{\br{0<\Re z\leq k\mu} \cap \br{z\not\in \Gamma_{\kappa}+k\mu-\epsilon }}\left\|\pa{L-[z-k\mu+\epsilon]I}^{-1}\right\|_{B(H_k)}<\infty.$$
Since $\Re z=-\epsilon$ is the closest vertical line to $\Re z=0$ which intersects\linebreak 
$\sigma\pa{\pa{L + [k\mu -\epsilon]I}\vert_{H_k}}$, we notice that $\br{0<\Re z\leq k\mu} \cap \br{z\not\in \Gamma_{\kappa}+k\mu-\epsilon} $ (repres\-en\-ted by the shaded area in Figure \ref{fig:compact}) is a compact set in the resolvent set of \linebreak $\pa{L+[k\mu-\epsilon]I}\vert_{H_k}$. As the resolvent map is analytic on the resolvent set, we conclude that $M_{k,\epsilon}<\infty$, completing the proof.
\end{proof}
\begin{remark}\label{rem:quantitative_Pruss_Gearhart}
While the constant mentioned in \eqref{eq:GPresult} is a fixed geometric one, the original Gearhart-Pr\"uss theorem doesn't give an estimation for it. However, recent studies have improved the original theorem and have managed to find explicit expression for this constant by paying a small price in the exponential power. As we can afford to ``lose'' another small $\epsilon$, we could use references such as \cite{HS,LV13} to have a more concrete expression for $C_{k,\epsilon}$. We will avoid giving such an expression in this work to simplify its presentation.
\end{remark}
We finally have all the tools to show Theorem \ref{thm:e2rateofdecay}:
\begin{proof}[Proof of Theorem \ref{thm:e2rateofdecay}]
Using the invariance of $V_0$ and $H_k$ under $L$ and Proposition \ref{prop:exp_stability} we find that for any \amit{$ f_k\in H_{k}$
\begin{eqnarray*}
&&e_2\pa{e^{Lt}\pa{ f_k+f_\infty}|f_\infty} =e_2\pa{e^{Lt}\pa{ f_k}+f_\infty|f_\infty}= \frac12\Norm{e^{Lt} f_k}^2_{H_{k}}\\
&&\leq \frac12 C^2_{k,\epsilon}e^{-2\pa{k\mu-\epsilon}t}\Norm{ f_k}^2_{H_{k}}=C^2_{k,\epsilon}e^{-2\pa{k\mu-\epsilon}t}e_2\pa{ f_k+f_\infty|f_\infty}\,,
\end{eqnarray*}}
showing the desired result.
\end{proof}
Theorem \ref{thm:e2rateofdecay} has given us the ability to control the rate of convergence to equilibrium of functions with initial data that, up to $f_\infty$, live on a ``higher eigen\-space''. Can we use this information to understand what happens to the solution of an arbitrary initial datum $f_0\in L^2\pa{\R^d,f_\infty^{-1}}$ with unit mass?\\
The answer to this question is \emph{Yes}.\\
Since for any $k\geq 1$
$$L^2\pa{\R^d,f_\infty^{-1}}=V_0\oplus\pa{\bigoplus_{m=1}^{k}V_m}\oplus H_{k+1}$$
and the Fokker-Planck semigroup is invariant under all the above spaces, we are motivated to \emph{split} the solution of our equation into a part in $V_0\oplus H_{k+1}$ and a part in $\bigoplus_{m=1}^{k}V_m$ - which is a \emph{finite dimensional subset of} $D(L)$. As we now know that decay in $\bigoplus_{m=1}^{k}V_m$ is slower than that for $H_{k+1}$ we will obtain a \emph{sharp} rate of convergence to equilibrium. We summarise the above intuition in the following theorem:
\begin{theorem}\label{thm:rate_of_decay_e2}
Consider the Fokker-Planck equation \eqref{eq:fokkerplanck} with diffusion and drift matrices satisfying Conditions (A)-(C). Let $f_0\in L^1_+\pa{\R^d}\cap L^2\pa{\R^d,f_\infty^{-1}}$ be a given function with unit mass \amit{such that
$$f_0=f_\infty+f_{k_0}+\tilde f_{k_0},$$
where $f_{k_0}\in V_{k_0}$ is non-zero and $\tilde f_{k_0}\in H_{k_0+1}$.
}Denote by $[L]_{k_0}$ the matrix representation of $L$ with respect to an orthonormal basis of $V_{k_0}$ and let 
%$$\mu_{k_0}=\min\br{\Re(\lambda)\;|\; \lambda\text{ is an eigen value of }[L]_{k_0}}=-k_0 \mu$$
$$n_{k_0}:=\max\br{\text{defect of }\lambda\;|\; \lambda\text{ is an eigenvalue of }[L]_{k_0}\text{ and }\Re \lambda=-k_0\mu},$$
where $\mu$ is defined in \eqref{eq:def_of_mu}. Then, there exists a geometric constant $c_{k_0}$, which is independent of $f_0$, such that
\begin{equation}\label{eq:rate_of_decay_e2_general}
e_2\pa{f(t)|f_\infty} \leq c_{k_0} e_2\pa{f_0|f_\infty}\pa{1+t^{2n_{k_0}}}e^{-2k_0\mu t}.
\end{equation}
\end{theorem}
\begin{remark}\label{rem:no_need_for_+}
As can be seen in the proof of the theorem, the sign of $f_0$ plays no role. As such, the theorem could have been stated for $f_0\in L^1\pa{\R^d}\cap L^2\pa{\R^d,f_\infty^{-1}}$. We decided to state it as is since it is the form we will use later on, and we wished to avoid possible confusion. 
\end{remark}
\begin{proof}[Proof of Theorem \ref{thm:rate_of_decay_e2}]
\amit{Due to the invariance of all $V_m$ under $L$ we see that
$$f(t)= f_\infty+ e^{Lt}f_{k_0}+e^{Lt}\tilde f_{k_0},$$
with $e^{Lt}f_{k_0}\in V_{k_0}$ and $e^{Lt}\tilde f_{k_0}\in H_{k_0+1}$.}
{}From Theorem \ref{thm:e2rateofdecay} we conclude that
$$e_2\pa{f_\infty+e^{Lt}\pa{\tilde f_{k_0}}|f_\infty} \leq c_{k_0,\epsilon}e_2\pa{f_\infty+\tilde f_{k_0}|f_\infty}e^{-2\pa{(k_0+1)\mu-\epsilon}t},$$
for any $0<\epsilon<\mu$.\\ 
Next, we denote by $d_k:=\text{dim}(V_k)$ and let $\br{\xi_{i}}_{i=1,\dots,d_{k_0}}$ be an orthonormal basis for $V_{k_0}$. The invariance of $V_m$ under $L$ implies that we can write
$$e^{Lt}f_{k_0}=\sum_{i=1}^{d_{k_0}}a_i(t)\xi_i$$
with $\bm{a}(t):
=\pa{a_1(t),\dots,a_{d_{k_0}}(t)}$ satisfying the simple ODE
\begin{equation*}%\label{eq:evolution_of_coef}
\dot{\bm{a}}(t)=[L]^T_{k_0}\bm{a}(t).
\end{equation*}
\\
%\textcolor{red}{Indeed, 
%$$\sum_{i=1}^{d_{k_0}}\dot{a}_i(t)\xi_i = \frac{d}{dt}\pa{e^{Lt}f_{k_0}}=L\pa{\sum_{i=1}^{d_{k_0}}a_i(t)\xi_i}=\sum_{i=1}^{d_{k_0}}a_i(t)\pa{\sum_{j=1}^{d_{k_0}}\pa{[L]_{k_0}}_{i,j}\xi_j}$$
%$$=\sum_{i=1}^{d_{k_0}}\pa{\sum_{j=1}^{d_{k_0}}\pa{[L]_{k_0}}_{j,i}a_j(t)}\xi_i$$
%which implies that
%$$\dot{a}_{i}(t)=\sum_{j=1}^{d_{k_0}}\pa{[L]^T_{k_0}}_{i,j}a_j(t).$$}
This, together with the definition of $n_{k_0}$ and the fact that a matrix and its transpose share eigenvalues and defect numbers, implies that we can find a geometric constant that depends only on $k_0$ such that
\begin{equation}\label{eq:rate_of_decay_e2_general_comp}
\sum_{i=1}^{d_{k_0}}a_i^2(t) \leq c_{k_0}\pa{1+t^{2n_{k_0}}}e^{-2k_0\mu t}\sum_{i=1}^{d_{k_0}}a_i^2(0).
\end{equation}
Since 
$$e_2\pa{f(t)|f_\infty} =e_2\pa{f_\infty+e^{Lt}(\tilde f_{k_0})+ e^{Lt}(f_{k_0})|f_\infty}=\frac{1}{2}\Norm{e^{Lt}(\tilde f_{k_0})+ e^{Lt}(f_{k_0})}^2_{L^2\pa{\R^d,f_\infty^{-1}}} $$
$$=\frac{1}{2}\Norm{e^{Lt}(\tilde f_{k_0})}^2_{L^2\pa{\R^d,f_\infty^{-1}}} + \frac{1}{2}\Norm{\sum_{i=1}^{d_{k_0}}a_i(t)\xi_i}^2_{L^2\pa{\R^d,f_\infty^{-1}}}$$
$$=e_2\pa{f_\infty+e^{Lt}(\tilde f_{k_0})|f_\infty}+\frac{1}{2}\sum_{i=1}^{d_{k_0}}a_i(t)^2,$$
we see, by combining Theorem \ref{thm:e2rateofdecay} and \eqref{eq:rate_of_decay_e2_general_comp} that 
\begin{equation}\nonumber
\begin{split}
e_2\pa{f(t)|f_\infty} \leq &c_{k_0,\epsilon}e_2\pa{f_\infty+\tilde f_{k_0}|f_\infty}e^{-2\pa{(k_0+1)\mu-\epsilon}t}\\
&+\frac{c_{k_0}}{2}\sum_{i=1}^{d_{k_0}}a_i^2(0)\pa{1+t^{2n_{k_0}}}e^{-2k_0\mu t}.
\end{split}
\end{equation}
%\begin{equation}\label{eq:exp_beats_poly}
%\max_{t\geq 0} e^{-\alpha t}t^{\beta}=e^{-\beta}\pa{\frac{\beta}{\alpha}}^{\beta}
%\end{equation}
%Using the fact that there exists a constant $E_{k_0}$ such that
%$$\sum_{i=1}^{d_{k_0}}a_i^2 \leq E_{k_0}\Norm{\sum_{i=1}^{d_{k_0}}a_i \xi_i}^2_{L^2\pa{\R^d,f_\infty^{-1}}},$$
%due to the equivalence of norm on finite dimensional spaces, we conclude that
Hence
\begin{equation*}
%\begin{split}
e_2\pa{f(t)|f_\infty} %& \\&
\leq \max\pa{c_{k_0,\epsilon},c_{k_0}} \pa{e_2\pa{f_\infty+\tilde f_{k_0}|f_\infty}+\frac{1}{2}\Norm{f_{k_0}}^2_{L^2\pa{\R^d,f_\infty^{-1}}}}\pa{1+t^{2n_{k_0}}}e^{-2k_0\mu t}.
%\end{split}
\end{equation*}
This completes the proof, as we have seen that 
$$e_2(f_0|f_\infty)=e_2\pa{f_\infty+\tilde f_{k_0}|f_\infty}+\frac{1}{2}\Norm{f_{k_0}}^2_{L^2\pa{\R^d,f_\infty^{-1}}}.$$
\end{proof}
\amit{
\begin{remark}\label{rem:no_split_for_e_p}
The idea to \emph{split} a solution into a few parts is viable \emph{only} for the $2-$entropy. The reason behind it is that such splitting, regardless of whether or not it can be done to functions outside of $L^2\pa{\R^d,f_\infty^{-1}}$, will most likely create functions without a definite sign. These functions can not be explored using the $p-$entropy with $1<p<2$.
\end{remark}
\amit{Theorem \ref{thm:rate_of_decay_e2} gives an optimal rate of decay for the $2-$entropy. However, one can underestimate the rate of decay by using Theorem \ref{thm:e2rateofdecay} and remove the condition $f_{k_0}\not=0$ to obtain the following:}
\begin{corollary}\label{corr:decay_est}
The statement of Theorem \ref{thm:rate_of_decay_e2} remains valid when replacing $k_0$ by any $1 \leq k_1 \leq k_0$. However, the decay estimate $\eqref{eq:rate_of_decay_e2_general}$ will not be sharp when $k_1<k_0$.
%The choice $k_1=1$ implies Theorem \ref{thm:main} for $p=2$.
\end{corollary}}
%\begin{remark}\label{rem:ODE_on_Vk}
%Looking at the above proof we notice that the evolution of the coefficients of $\xi_i$ on $V_1$, given by \eqref{eq:evolution_of_coef} is
%$$\dot{\bm{a}}(t)=-\C\bm{a}(t).$$
%This is exactly the reason behind the identical rate of convergence to equilibrium in our Fokker-Planck equation and the ODE associated to $\C$.
%\end{remark}
\amit{\begin{proof}[Proof of Theorem \ref{thm:main} for $p=2$] The proof follows immediately from Coro\-llary \ref{corr:decay_est} for $k_1=1$.
\end{proof}}
Now that we have learned everything we can on the convergence to equilibrium for $e_2$, we can proceed to understand the convergence to equilibrium of $e_p$.

%%%%%%%%%%%%%%%%%%%%%%%%%%%%%%%%%%%%%%%%%%%%%%%%%%%%%%%%%%%%%%%%%%%%%%%%%%%%%%%%%%%%%%%%%%%%%%%%%%%%%%%%%%%%%%%%%
%%%%%%%%%%%%%%%%%%%%%%%%%%%%%%%%%%%%%%%%%%%%%%%%%%%%%%%%%%%%%%%%%%%%%%%%%%%%%%%%%%%%%%%%%%%%%%%%%%%%%%%%%%%%%%%%%

\section{Non-symmetric Hypercontractivity and Rates of Convergence for the $p-$Entropy}\label{sec:hyper}
In this section we will show how to deduce the rate of convergence to equilibrium for the family of $p-$entropies, with $1<p<2$, from $e_2$. The main thing that will make the above possible is \emph{a non-symmetric hypercontractivity} property of our Fokker-Planck equation - namely, that any solution to the equation with (initially only) a finite $p-$entropy will eventually be ``pushed'' into $L^2\pa{\R^d,f_\infty^{-1}}$, at which point we can use the information we gained on $e_2$. \\
Before we show this result, and see how it implies our main theorem, we explain why and how this non-symmetric hypercontractivity helps.
\begin{lemma}\label{lem:control_of_ep_by_e2}
Let $f\in L^1_+\pa{\R^d}$ with unit mass. Then
\begin{enumerate}[(i)]
\item \begin{equation*}%\label{eq:formula_for_ep}
e_p(f|f_\infty)=\frac{1}{p(p-1)}\pa{\norm{f}^p_{L^p\pa{\R^d,f_\infty^{1-p}}}-1}.
\end{equation*}
\item for any $1<p_1< p_2 \leq 2 $ there exists a constant $C_{p_1,p_2}>0$ such that
\begin{equation*}%\label{eq:fcontrol_of_ep1_by_ep2}
e_{p_1}(f|f_\infty)\leq C_{p_1,p_2}e_{p_2}(f|f_\infty).
\end{equation*}
In particular, for any $1<p<2$
\begin{equation*}%\label{eq:control_of_ep_by_e2}
e_{p}(f|f_\infty)\leq C_{p}e_{2}(f|f_\infty),
\end{equation*}
for a fixed geometric constant.
\end{enumerate}
\end{lemma}
\begin{proof}
%To prove $(i)$ we notice that since $f$ is of unit mass
%$$\int_{\R^d}\pa{\frac{f(x)}{f_\infty(x)}-1}f_\infty(x)dx=0.$$
%Thus 
%$$e_p(f|f_\infty)=\frac{1}{p(p-1)}\int_{\R^d}\pa{\pa{\frac{f(x)}{f_\infty(x)}}^p - 1 }f_\infty(x)dx=\frac{1}{p(p-1)}\pa{\norm{f}^p_{L^p\pa{\R^d,f_\infty^{1-p}}}-1}.$$
$(i)$ is trivial. 
To prove $(ii)$ we consider the function 
$$g(y):=\begin{cases}
\frac{p_2(p_2-1)}{p_1(p_1-1)}\frac{y^{p_1}-p_1(y-1)-1}{y^{p_2}-p_2(y-1)-1}\,, & y\geq 0, y\not=1 \\
1\,, & y=1.
\end{cases}$$
Clearly $g \geq 0$ on $\R^+$, and it is easy to check that it is continuous. Since we have $\lim_{y\rightarrow\infty}g(y)=0$, we can conclude the result using \eqref{eq:defentropy}. 
\end{proof}
It is worth to note that the second point of part $(ii)$ of Lemma \ref{lem:control_of_ep_by_e2} can be extended to general generating function for an admissible relative entropy. The following is taken from \cite{AMTU01}:
\begin{lemma}\label{lem:control_of_psi_by_e2}
Let $\psi$ be a generating function for an admissible relative entropy. Then one has that
\begin{equation*}%\label{eq:control_of_psi_by_e2}
\psi(y) \leq 2\psi^{\prime\prime}(1)\psi_2(y),\quad y\geq 0.
\end{equation*}
In particular $e_p \leq 2 e_2$ for any $1<p<2$ whenever $e_2$ is finite. 
\end{lemma}
Lemma \ref{lem:control_of_ep_by_e2} assures us that, if we start with initial data in $L^2\pa{\R^d,f_\infty^{-1}}$, then $e_p$ will be finite. Moreover, due to Theorem \ref{thm:main} for $p=2$, and the fact that the solution to \eqref{eq:fokkerplanck} remains in $L^2\pa{\R^d,f_\infty^{-1}}$, we have that
$$e_p(f(t)|f_\infty) \leq 2e_2(f(t)|f_\infty) \leq Ce_2(f_0|f_\infty)\pa{1+t^{2n}}e^{-2\mu t}. $$
However, one can easily find initial data $f_0\not\in L^2\pa{\R^d,f_\infty^{-1}}$ with finite $p-$entropies. If one can show that the flow of the Fokker-Planck equation eventually forces the solution to enter $L^2\pa{\R^d,f_\infty^{-1}}$, we would be able to utilise the idea we just presented, at least from that time on.\\
This \emph{explicit non-symmetric} hypercontractivity result we desire, is the main new theorem we present in this section.
\begin{theorem}\label{thm:hyper}
Consider the Fokker-Planck equation \eqref{eq:fokkerplanck} with diffusion and drift matrices $\D$ and $\C$ satisfying Conditions (A)-(C). Let $f_0 \in L^1_+\pa{\R^d}$  be a function with unit mass and assume there exists $\epsilon>0$ such that
\begin{equation}\label{eq:hypercondition} 
\int_{\R^d}e^{\epsilon \abs{x}^2}f_0(x)dx <\infty.
\end{equation}
\begin{enumerate}
\item[(i)]
Then, for any $q>1$, there exists an explicit $t_0>0$ that depends only on geometric constants of the problem such that the solution to \eqref{eq:fokkerplanck} satisfies
\begin{equation}\label{eq:hyper}
\int_{\R^d}f(t,x)^q f_\infty^{-1}(x)dx \leq \pa{\frac{q}{\pi (q+1)}}^{\frac{qd}{2}}\pa{\frac{8\pi^2}{q-1}}^{\frac{d}{2}} \pa{\int_{\R^d}e^{\epsilon\abs{x}^2}f_0(x)dx}^q
\end{equation}
for all $t\geq t_0$. 
\item[(ii)] In particular, if $f_0$ satisfies $e_p(f_0|f_\infty)<\infty$ for some $1<p<2$ we have that
\begin{equation}\label{eq:entropichyper}
\begin{split}
e_2(f(t)|f_\infty)\leq \frac{1}{2}\pa{\pa{\frac{8\sqrt{2}}{3\cdot 2^{\frac{1}{p}}}}^d \pa{p(p-1)e_p(f_0|f_\infty)+1}^{\frac{2}{p}}-1},
\end{split}
\end{equation}
for $t\geq \tilde{t}_0(p)>0$, which can be given explicitly.
\end{enumerate}
\end{theorem}
\amit{\begin{remark}\label{rem:improvement_of_constants}
As we consider $e_p$ in our hypercontractivity, which is, up to a constant, the $L^p$ norm of $g:=\frac{f}{f_\infty}$ with the measure $f_\infty(x)dx$, one can view our result as a hypercontractivity property of the Ornstein-Uhlenbeck operator, $P$ (for an appropriate choice of the diffusion matrix $\Q$ and drift matrix $\B$), discussed in \S\ref{sec:spectral}. With this notation, \eqref{eq:entropichyper} is equivalent to
\begin{equation}\label{eq:reformulation_of_hyper}
\|g(t)\|_{L^2(f_\infty)} \leq C_{p,d} \|g_0\|_{L^p(f_\infty)}\,,\quad t\geq \tilde{t}_0(p)
\end{equation}
for $1<p<2$, where $C_{p,d}:=\left(\frac{8 \sqrt{2}}{3\cdot 2^{\frac{1}{p}}} \right)^{\frac{d}{2}}$. Since $e_2$ decreases along the flow of our equation, \eqref{eq:reformulation_of_hyper} is valid for $p=2$ with $C_{2,d}=1$. Thus, by using the Riesz-Thorin theorem one can improve inequality \eqref{eq:reformulation_of_hyper} to the same inequality with the constant $C_{p,d}^{\frac{2}{p}-1}$. We would like to point out at this point that a simple limit process shows that \eqref{eq:reformulation_of_hyper} is also valid for $p=1$, but there is no connection between the $L^1$ norm of $g$ and the Boltzmann entropy, $e_1$, of $f_0$.  
\end{remark}}

\amit{\begin{remark}\label{rem:Wang}
Since its original definition for the Ornstein-Uhlenbeck semigroup in the work of Nelson, \cite{Ne73}, the notion of \emph{hypercontractivity} has been studied extensively for Markov diffusive operators (implying selfadjointness). A contemporary review of this topic can be found in \cite{BaGeLe14}. For such selfadjoint generators, hypercontractivity is equivalent to the validity of a logarithmic Sobolev inequality, as proved by Gross \cite{Gro75}. For non-symmetric generators, however, this equivalence does not hold: While a log Sobolev inequality still implies hypercontractvity of related semigroups (cf.\ the proof of Theorem 5.2.3 in \cite{BaGeLe14}), the reverse implication is not true in general (cf. Remark 5.1.1 in \cite{W05}). In particular, hypocoercive degenerate parabolic equations cannot give rise to a log Sobolev inequality, but they may exhibit hypercontractivity (as just stated above).\\
The last 20 years have seen the emergence of the, more delicate, study of hypercontractivity for non-symmetric and even degenerate semigroups. Notable works in the field are the paper of Fuhrman, \cite{F98}, and more recently the work of Wang et al., \cite{BWY15SPA,BWY15EJP,W17}. Most of these works consider an abstract Hilbert space as an underlying domain for the semigroup, and to our knowledge none of them give an explicit time after which one can observe the hypercontractivity phenomena (Fuhrman gives a condition on the time in \cite{F98}).\\
Our hypercontractivity theorem, which we will prove shortly, gives not only an explicit and quantitative inequality, but also provides an estimation on the time one needs to wait before the hypercontractivity occurs. To keep the formulation of Theorem \ref{thm:hyper} simple we did not include this ``waiting time'' there, but we emphasised it in its proof. Moreover, the hypercontractivity estimate from Theorem \ref{thm:hyper}(i) only requires \eqref{eq:hypercondition}, a weighted $L^1$ norm of $f_0$. This is weaker than in usual hypercontractivity estimates, which use $L^p$ norms as on the r.h.s. of \eqref{eq:reformulation_of_hyper}.
\end{remark}} 
It is worth to note that we prove our theorem under the setting of the $e_p$ entropies, which can be thought of as $L^p$ spaces with a weight function that depends on $p$.

In order to be able to prove Theorem \ref{thm:hyper} we will need a few technical lemmas. 
\begin{lemma}\label{lem:exact_sol}
Given $f_0\in L^1_+\pa{\R^d}$ with unit mass, the solution to the Fokker-Planck equation \eqref{eq:fokkerplanck} with diffusion and drift matrices $\D$ and $\C$ that satisfy Conditions (A)-(C) is given by
\begin{equation}\label{eq:exactsolution}
f(t,x)=\frac{1}{\pa{2\pi}^{\frac{d}{2}}\sqrt{\det \W(t)}}\int_{\R^d}e^{-\frac{1}{2}\pa{x-e^{-\C t}y}^T \W(t)^{-1}\pa{x-e^{\C t}y}} f_0(y)dy,
\end{equation} 
where 
\begin{equation*}%\label{eq:W}
\W(t):=2\int_{0}^t e^{-\C s}\D e^{-\C^T s}ds.
\end{equation*}
\end{lemma}  
This is a well known result, see for instance \S 1 in \cite{Ho67} or \S 6.5 in \cite{RiFP89}.

\begin{lemma}\label{lem:WconvergesKrate}
Assume that the diffusion and drift matrices, $\D$ and $\C$, satisfy Conditions (A)-(C), and let $\K$ be the unique positive definite matrix that satisfies
\begin{equation*}%\label{eq:lyapunov}
2\D=\C\K+\K\C^T.
\end{equation*}
Then (in any matrix norm) 
\begin{equation*}%\label{eq:WconvergesKrate}
\norm{\W(t)-\K} \leq  c  (1+t^{2n})e^{-2\mu t}, \quad t\geq 0,
\end{equation*}
where $c>0$ is a geometric constant depending on $n$ and $\mu$, with $n$ being the maximal defect of the eigenvalues of $\C$ with real part $\mu$, defined in \eqref{eq:def_of_mu}.
\end{lemma}
\begin{proof}
We start the proof by noticing that $\K$ is given by 
$$\K=2\int_{0}^\infty e^{-\C s}\D e^{-\C^T s}ds$$
(see for instance \cite{OPPS15}). As such
$$\norm{\W(t)-\K} \leq 2\int_{t}^{\infty} \norm{e^{-\C s}\D e^{-\C^T s}}ds \leq 2\norm{\D}\int_{t}^{\infty} \norm{e^{-\C s}} \norm{e^{-\C^T s}}ds.$$
Using the fact that $${\bf{A}} e^{-\C t}{\bf{A}}^{-1} = e^{-\A\C\A^{-1}t}$$ for any regular matrix $\A$, we conclude that, if ${\bf{J}}$ is the Jordan form of $\C$, then
\begin{equation}\label{eq:exp_C_connection_jordan}
\norm{e^{-\C t}} \leq \norm{\A_{{\bf{J}}}}\norm{\A_{{\bf{J}}}^{-1}} \norm{e^{-{\bf{J}}t}}\,,
\end{equation}
where $\A_{{\bf{J}}}$ is the similarity matrix between $\C$ and its Jordan form.\\
For a single Jordan block of size $n+1$ (corresponding to a defect of $n$ in the eigenvalue $\lambda$), ${\bf{\widetilde{J}}}$, we find that
$$e^{{\bf{\widetilde{J}}}t}=\pa{\begin{tabular}{ccccc}
$e^{\lambda t}$ & $t e^{\lambda t}$ & $\dots$ & & $\frac{t^n}{n!}e^{\lambda t}$ \\
 & $e^{\lambda t}$ & $\ddots$ & &$\frac{t^{n-1}}{(n-1)!}e^{\lambda t}$ \\
 &  & $\ddots$ & &$\vdots$ \\
$0$ &  &  & &$e^{\lambda t}$
\end{tabular}} \qquad \text{ where }\qquad 
{\bf{\widetilde{J}}}=\pa{\begin{tabular}{cccc}
$\lambda$ & $1$ & & $0$ \\
  &  $\ddots$  & $\ddots$ &  \\
  &&&1\\
$0$ &  &  & $\lambda$
\end{tabular}} .$$
Thus, we conclude that
$$\norm{e^{{\bf{\widetilde{J}}}t}x}_{1} \leq \sum_{i=1}^{n+1} \sum_{j=i}^{n+1} \frac{t^{j-i}}{(j-i)!}e^{\Re(\lambda) t} \abs{x_j} \leq \pa{ \sum_{i=1}^{n+1} \pa{1+t^n}e^{\Re(\lambda) t}}\norm{x}_{1} $$
$$=(n+1)\pa{1+t^n}e^{\Re(\lambda) t}\norm{x}_{1},\quad t\geq 0.$$
Due to the equivalence of norms on finite dimensional spaces, there exists a geometric constant $c_1>0$, that depends on $n$, such that
\begin{equation}\label{eq:jordan_bound}
\norm{e^{{\bf{\widetilde{J}}}t}} \leq c_1\pa{1+t^n}e^{\Re(\lambda) t}.
\end{equation}
Coming back to $\C$, we see that the above inequality together with \eqref{eq:exp_C_connection_jordan} imply that $\norm{e^{-\C t}}$ is controlled by the norm of $\C$'s largest (measured by the defect number) Jordan block of the eigenvalue with smallest real part. From this, and \eqref{eq:jordan_bound}, we conclude that
\begin{equation}\label{eq:C-convergence}
\|e^{-\C t} \| \leq c_2 (1+t^{n})e^{-\mu t}, \quad t\geq 0.
\end{equation} 
The same estimation for $\norm{e^{-\C^T t}}$ implies that
$$\norm{\W(t)-\K}\leq c_3\int_{t}^{\infty} \pa{1+s^{2n}}e^{-2\mu s}ds,$$
for some geometric constant $c_3>0$ that depends on $n$. Since
\[
\int^\infty_t s^{2n} e^{-2\mu s} ds = \left[\frac{1}{2\mu} t^{2n}  + \frac{2n}{(2\mu)^2} t^{2n-1}  + \frac{2n(2n-1)}{(2\mu)^3} t^{2n-2} +...+ \frac{(2n)!}{(2\mu)^{2n+1}}\right]e^{-2\mu t}
\]
we conclude the desired result.
\end{proof}
While we can continue with a general matrix $\K$, it will simplify our computations greatly if $\K$ would have been $\II$. Since we are working under the assumption that $\D=\C_S$, the normalization from Theorem \ref{thm:simpliedDC} implies exactly that. Thus, from this point onwards we will assume that $\K$ is $\II$.
\begin{lemma}\label{lem:Inverse}
For any $\epsilon>0$ there exists an explicit $t_1>0$ such that for all $t\geq t_1$
\begin{equation*}%\label{eq:W-1convergence}
\norm{\W^{-1}(t)-\II} \leq \epsilon,
\end{equation*}
where $\W(t)$ is as in Lemma \ref{lem:WconvergesKrate}. An explicit, but not optimal choice for $t_1$ is given by
\begin{equation}\label{eq:explicit_time}
t_1(\epsilon): =\frac{1}{2(\mu-\alpha)} \log\pa{\frac{c(1+\epsilon)\pa{1+\pa{\frac{n}{\alpha e}}^{2n}}}{\epsilon}},
\end{equation}
where $0<\alpha<\mu$ is arbitrary and $c>0$ is given by Lemma \ref{lem:WconvergesKrate}.
\end{lemma}
\begin{proof}
We have that for any invertible matrix $\A$ 
$$\|\A^{-1}-\II \|= \|\pa{\A-\II}\A^{-1}\| \leq \norm{\A-\II}\norm{\A^{-1}}.$$
In addition, if $\norm{\A-\II}<1$, then
$$\norm{\A^{-1}}= \norm{\pa{\II-\pa{\II-\A}}^{-1}} \leq \frac{1}{1-\norm{\A-\II}}.$$
Thus, for any $t>0$ such that $\norm{\W(t)-\II}<1$ we have that
\begin{equation}\label{eq:W-1est}
\norm{\W^{-1}(t)-\II} \leq \frac{\norm{\W(t)-\II}}{1-\norm{\W(t)-\II}}.
\end{equation}
Defining $\tilde{t_1}(\epsilon)$ as 
\begin{equation}\label{eq:t1tilde}
\tilde{t_1}(\epsilon):=\min\br{s\geq 0\,  \bigg|\,\pa{1+t^{2n}}e^{-2\mu t} \leq \frac{\epsilon}{c(1+\epsilon)},\quad \forall t\geq s},
\end{equation}
with the constant $c$ given by Lemma \ref{lem:WconvergesKrate}, we see from Lemma \ref{lem:WconvergesKrate} that for any $t\geq \tilde{t_1}(\epsilon)$
$$\norm{\W(t)-\II} \leq \frac{\epsilon}{1+\epsilon}.$$
Combining the above with \eqref{eq:W-1est}, shows the first result for $t_1= \tilde{t_1}(\epsilon)$.

\medskip
To prove the second claim we will show that  
\begin{equation*}%\label{eq:t1-t1tilde}
t_1(\epsilon)\geq \tilde{t_1}(\epsilon).
\end{equation*}
For this elementary proof we use the fact that
$$\max_{t\geq 0} e^{-a t}t^b = \pa{\frac{b}{a e}}^{b}$$
for any $a,b >0$. Thus, choosing $a = 2\alpha$, where $0<\alpha<\mu$ is arbitrary, and $b=2n$ we have that 
$$\pa{1+t^{2n}}e^{-2\mu t} \leq \pa{1+\pa{\frac{n}{\alpha e}}^{2n}} e^{-2(\mu-\alpha)t},\quad t\geq 0.$$
As a consequence, if 
\begin{equation}\label{eq:expest}
\pa{1+\pa{\frac{n}{\alpha e}}^{2n}} e^{-2(\mu-\alpha)t} \leq \frac{\epsilon}{c(1+\epsilon)},\quad \forall t\geq s,
\end{equation}
then $s\geq \tilde{t_1}(\epsilon)$ due to \eqref{eq:t1tilde}. The smallest possible $s$ in \eqref{eq:expest} is obtained by solving the corresponding equality for $t$, and yields \eqref{eq:explicit_time}, concluding the proof. 
\end{proof}
We now have all the tools to prove Theorem \ref{thm:hyper}
\begin{proof}[Proof of Theorem \ref{thm:hyper}]
To show $(i)$ we recall Minkowski's integral inequality, which will play an important role in estimating the $L^p$ norms of $f(t)$. \\
\textbf{Minkowski's Integral Inequality:} \emph{For any non-negative measurable function $F$ on $(X_1\times X_2, \mu_1\times \mu_2)$, and any $q\geq 1$ one has that}
\begin{equation}\label{eq:Mink}
\begin{split}
\left(\int_{X_2}\abs{\int_{X_1}F(x_1,x_2)d\mu_1( x_1)}^q  d\mu_2(x_2) \right)^{\frac{1}{q}}& \\
\leq \int_{X_1} \pa{ \int_{X_2}\abs{F(x_1,x_2)}^q d\mu_{2}( x_2) }^{\frac{1}{q}}&d\mu_1(x_1).
\end{split}
\end{equation}
Next, we fix an $\epsilon_1=\epsilon_1(\epsilon,q)\in(0,1)$, to be chosen later. From Lemma \ref{lem:WconvergesKrate} and \ref{lem:Inverse} we see that, for $t\geq t_1(\epsilon_1)$ with 
$$t_1(\epsilon_1): =\frac{1}{2(\mu-\alpha)} \log\pa{\frac{c(1+\epsilon_1)\pa{1+\pa{\frac{n}{\alpha e}}^{2n}}}{\epsilon_1}}$$
for some fixed $0<\alpha<\mu$, we have that
$$\norm{\W(t)-\II} \leq \frac{\epsilon_1}{1+\epsilon_1}<\epsilon_1,\qquad \norm{\W^{-1}(t)-\II} \leq \epsilon_1,$$
and hence
$$
\W(t) > (1-\epsilon_1)\II,\qquad \W(t)^{-1}\geq (1-\epsilon_1)\II.
$$
As such, for $t\geq t_1(\epsilon_1)$
\begin{equation}\label{eq:hyperproofI}
\abs{e^{-\frac{1}{2}\pa{x-e^{-\C t}y}^T \W(t)^{-1}\pa{x-e^{\C t}y}} f_0(y)}^q \leq e^{-\frac{q}{2}(1-\epsilon_1)\abs{x-e^{-\C t}y}^2} \abs{f_0(y)}^{q}
\end{equation}
and
\begin{equation}\label{eq:hyperproofII}
\det{\W(t)} \geq (1-\epsilon_1)^d.
\end{equation}
We conclude, using \eqref{eq:Mink}, the exact solution formula \eqref{eq:exactsolution}, \eqref{eq:hyperproofI} and \eqref{eq:hyperproofII} that for $t\geq t_1(\epsilon_1)$ it holds:
\begin{equation}\label{eq:hyperproofIII}
\begin{split}\int_{\R^d} &\abs{f(t,x)}^q f_\infty^{-1}(x)dx \\\leq& \frac{(2\pi)^{\frac{d}{2}}}{\pa{2\pi (1-\epsilon_1)}^{\frac{qd}{2}}}\pa{\int_{\R^d} \pa{\int_{\R^d}e^{-\frac{q}{2}(1-\epsilon_1)\abs{x-e^{-\C t}y}^2} \abs{f_0(y)}^{q} e^{\frac{\abs{x}^2}{2}}dx}^{\frac{1}{q}}dy}^q \\
=&\frac{(2\pi)^{\frac{d}{2}}}{\pa{2\pi (1-\epsilon_1)}^{\frac{qd}{2}}}\pa{\int_{\R^d}\pa{\int_{\R^d}e^{-\frac{q}{2}(1-\epsilon_1)\abs{x-e^{-\C t}y}^2} e^{\frac{\abs{x}^2}{2}}dx}^{\frac{1}{q}} \abs{f_0(y)}dy}^q.
\end{split}
\end{equation}
We proceed by choosing $\epsilon_1>0$ such that $q(1-\epsilon_1)>1$ (or equivalently $\epsilon_1 < \frac{q-1}{q}$) and denoting
$$\eta:=q(1-\epsilon_1)-1>0.$$
Shifting the $x$ variable by $\frac12 e^{-\C t}y$ and completing the square, we find that
\begin{equation}\label{eq:hyperproofIV}
\begin{split}
\int_{\R^d}e^{-\frac{q}{2}(1-\epsilon_1)\abs{x-e^{-\C t}y}^2} e^{\frac{\abs{x}^2}{2}}dx &= \int_{\R^d}e^{-\frac{\eta+1}{2}\abs{x-\frac{1}{2}e^{-\C t}y}^2} e^{\frac{\abs{x+\frac{1}{2}e^{-\C t}y}^2}{2}}dx\\
=\int_{\R^d} e^{x e^{-\C t}y}e^{-\frac{\eta}{2}\abs{x-\frac{1}{2}e^{-\C t}y}^2}dx&=\int_{\R^d}e^{-\frac{\eta}{2}\abs{x-\frac{1}{2}\pa{1+\frac{2}{\eta}}e^{-\C t}y}^2}e^{\pa{\frac12+\frac{1}{2\eta}}\abs{e^{-\C t}y}^2}dx\\
&=\pa{\frac{2\pi}{\eta}}^{\frac{d}{2}} e^{\pa{\frac12+\frac{1}{2\eta}}\abs{e^{-\C t}y}^2}.
\end{split}
\end{equation}
Using \eqref{eq:C-convergence} we can find a uniform geometric constant $c_2$ such that 
\begin{equation}\nonumber
\|e^{-\C t}\|^2\leq c_2^2 \pa{1+t^{n}}^2 e^{-2\mu t}\leq 2c_2^2 \pa{1+t^{2n}}e^{-2\mu t}.
\end{equation}
Following the proof of Lemma \ref{lem:Inverse} we recall that if 
$$t\geq \frac{1}{2(\mu-\alpha)}\log\pa{\frac{\tilde{c}(1+\epsilon_2)\pa{1+\frac{n}{\alpha e}}^{2n}}{\epsilon_2}},$$
where $0<\alpha<\mu$ is arbitrary and for any $\tilde{c},\epsilon_2>0$, then
$$\pa{1+t^{2n}}e^{-2\mu t} \leq \frac{\epsilon_2}{\tilde{c}(1+\epsilon_2)}.$$
Thus, choosing 
$$\tilde{c}=\frac{c_2^2(1+\eta)}{q\eta}=\frac{c_2^2(1-\epsilon_1)}{q(1-\epsilon_1)-1}\quad \text{ and }\quad \epsilon_2=\frac{\epsilon_1}{1-\epsilon_1}$$
we get that if 
\begin{equation*}%\label{eq:def_of_t2}
t\geq t_2(\epsilon_1):=\frac{1}{2(\mu-\alpha)}\log\pa{\frac{c_2^2(1-\epsilon_1)\pa{1+\frac{n}{\alpha e}}^{2n}}{\pa{q(1-\epsilon_1)-1}\epsilon_1}},
\end{equation*}
where $0<\alpha<\mu$ is arbitrary and for any $\tilde{c},\epsilon_2>0$, then
\begin{equation}\nonumber
\pa{\frac{1}{2}+\frac{1}{2\eta}}\|e^{-\C t}\|^2\leq \frac{c_2^2(1+\eta)}{q\eta}q\pa{1+t^{2n}}e^{-2\mu t} \leq q\epsilon_1.
\end{equation}
% $t_1(\epsilon) \geq \tilde{t_1}(\epsilon)$ from it (given by \eqref{eq:explicit_time} and \eqref{eq:t1tilde}) we conclude that if $t\geq t_1(\epsilon_2)$ 
%$$\|e^{-\C t}\|^2\leq 2c_2^2 \frac{\epsilon_2}{c(1+\epsilon_2)}.$$
%Choosing 
%$$\epsilon_2=$$
%}
%\\
%{}From \eqref{eq:t1tilde}and \eqref{eq:t1-t1tilde} we see that \eqref{eq:C-convergence} is uniformly bounded by
%$$\|e^{-\C t}\|^2\leq c_2 \frac{\epsilon}{c(1+\epsilon)}, \quad \forall t\geq t_1(\epsilon)\geq \tilde{t_1}(\epsilon).$$
%Hence we obtain
%$$
%\|e^{-\C t}\|^2 \leq \frac{2[q(1-\epsilon_1)-1]}{1-\epsilon_1}
%$$
%for 
%$$
%t\geq t_2(\epsilon_1):=\frac{1}{2(\mu-\alpha)}\log \left(\frac{c_2(1-\epsilon_1)\left(1+\left(\frac{n}{\alpha e}\right)^{2n}\right)}{2\epsilon_1[q(1-\epsilon_1)-1]} \right)
%$$
%for any fixed $0<\alpha<\mu$, and thus 
%$$
%\left(\frac12 +\frac{1}{2\mu} \right)\left|e^{-\C t}y \right|^2\leq q\epsilon_1 |y|^2
%$$
%for $t\geq t_2(\epsilon_1)$.\\
Combining this with our previous computations (\eqref{eq:hyperproofIII} and \eqref{eq:hyperproofIV}), we find that for any $t\geq t_0(\epsilon_1):=\max\pa{t_1(\epsilon_1),t_2(\epsilon_1)}$
$$\int_{\R^d}\abs{f(t,x)}^q f_\infty^{-1}(x)dx \leq \frac{(2\pi)^{d(1-\frac{q}{2})}}{(1-\epsilon_1)^{\frac{qd}{2}}\eta^{\frac{d}{2}}} \pa{\int_{\R^d}e^{\epsilon_1\abs{y}^2}f_0(y)dy}^q.$$

If $\epsilon_1$ is chosen more restrictively than before, namely  $\epsilon_1 \leq \frac{q-1}{2q}$, then we have 
$$ \frac{q-1}{2} \leq \eta< q-1 \qquad \text{and} \qquad 1-\epsilon_1 \geq \frac{q+1}{2q},$$
which implies the first statement of the theorem by choosing $\epsilon_1:=\min\pa{\epsilon,\frac{q-1}{2q}}$.

\medskip
For the proof of (ii) we note that \eqref{eq:entropichyper} is equivalent to
\begin{equation}\label{eq:hyperequiv}
\|f(t)\|^2_{L^2\pa{\R^d,f_\infty^{-1}}} \leq\pa{\frac{8\sqrt{2}}{3\cdot 2^{\frac{1}{p}}}}^d\|f_0\|^2_{L^p\pa{\R^d,f_\infty^{1-p}}}.
\end{equation}
With the H\"{o}lder inequality we obtain
\begin{equation*}
\begin{split}
\int_{\R^d} e^{\frac{p-1}{4p}|x|^2}f_0(x) dx &\leq \left(\int_{\R^d} e^{-\frac{|x|^2}{4}} dx \right)^{\frac{p-1}{p}} \left(\int_{\R^d} e^{\frac{p-1}{2}|x|^2}f_0^p(x)dx \right)^{\frac1p}\\
&= 2^{\frac{d}{2}\frac{p-1}{p}}\|f_0\|_{L^p\pa{\R^d,f_\infty^{1-p}}}.
\end{split}
\end{equation*}
Hence, $e_p(f_0|f_\infty)<\infty$ implies \eqref{eq:hypercondition} with $\epsilon=\frac{p-1}{4p}$, and \eqref{eq:hyperequiv} follows from \eqref{eq:hyper} with $q=2$ and $\tilde t_0(p)=t_0\pa{\frac{p-1}{4p}}$.
%for any $p>1$
%$$\int_{\R^d} e^{\frac{\epsilon}{2p}\abs{y}^2}f_0(y)dy =\int_{\R^d} e^{-\frac{\epsilon}{2p}\abs{y}^2}e^{\frac{\epsilon}{p}\abs{y}^2}f_0(y)dy$$
%$$\leq \pa{\int_{\R^d}e^{-\frac{\epsilon}{2(p-1)}\abs{y}^2}dy}^{\frac{p-1}{p}}\pa{\int_{\R^d}e^{\epsilon \abs{y}^2}f_0(y)^p dy}^{\frac{1}{p}}
%=\pa{\frac{2\pi(p-1)}{\epsilon}}^{\frac{d}{2}}\pa{\int_{\R^d}e^{\epsilon \abs{y}^2}f_0(y)^p dy}^{\frac{1}{p}}.$$
%Thus, using \eqref{eq:hyper} for $L^2$ we see that for any $1<p<2$ and $t\geq t_0\pa{2p\epsilon}$
%\begin{equation}\label{eq:hyper_for_ep_step_I}
%\begin{gathered}
%\int_{\R^d}f(t,x)^2 f_\infty^{-1}(x)dx \leq \pa{\frac{2}{3\pi}}^{d}\pa{4\pi}^{\frac{d}{2}} \pa{\int_{\R^d}e^{\epsilon\abs{y}^2}f_0(y)dy}^2\\
%\leq \pa{\frac{2}{3\pi}}^{d}\pa{4\pi}^{\frac{d}{2}}\pa{\frac{\pi(p-1)}{p\epsilon}}^{d}\pa{\int_{\R^d}e^{2p\epsilon \abs{y}^2}f_0(y)^p dy}^{\frac{2}{p}}
%\end{gathered}
%\end{equation}
%As we've seen in Lemma \ref{lem:control_of_ep_by_e2}
%$$e_p(f|f_\infty)=\frac{1}{p(p-1)}\pa{\int_{\R^d}f(x)^p e^{\frac{p-1}{2}\abs{x}^2}dx-1},$$
%so by choosing $\epsilon=\frac{p-1}{4p}$ we find an explicit time $t_0(p)$ such that for any $t\geq t_0(p)$
%\begin{equation}\label{eq:hyper_for_ep_step_II}
%\begin{gathered}
%e_2\pa{f(t)|f_\infty}=\frac{1}{2}\pa{\int_{\R^d}f(t,x)^2 f_\infty^{-1}(x)dx-1} \\ 
%\leq \frac{1}{2}\pa{\pa{\frac{2}{3\pi}}^{d}\pa{4\pi}^{\frac{3d}{2}}\pa{p(p-1)e_p(f_0|f_\infty)+1)}^{\frac{2}{p}}-1}.
%\end{gathered}
%\end{equation}
%This concludes the proof.
\end{proof}
\begin{remark}\label{rem:explicit_time_hyper}
If the condition \eqref{eq:hypercondition} holds for $\epsilon=\frac12$ we can give an explicit upper bound for the ``waiting time'' in the hypercontractivity estimate \eqref{eq:hyper}. For such $\epsilon$ we have $\epsilon_1:=\min\pa{\epsilon,\frac{q-1}{2q}}=\frac{q-1}{2q}$, and by choosing $\alpha=\frac{\mu}{2}$ we can see that $t_0(\epsilon_1)$ from the proof of Theorem \ref{thm:hyper} is   
$$\overline{t_0}(q): =\frac{1}{\mu} \log\pa{\frac{\max\pa{c(3q-1),2c_2^2\frac{q+1}{q-1}} \pa{1+\pa{\frac{2n}{\mu e}}^{2n}}}{q-1 }},
$$
where $c,c_2$ are geometric constants found in the proof of Lemma \ref{lem:WconvergesKrate}. 
\end{remark}
With the non-symmetric hypercontractivity result at hand, we can finally complete the proof of our main theorem for $1<p<2$.
\begin{proof}[Proof of Theorem \ref{thm:main} for $1<p<2$]
Using Theorem \ref{thm:hyper} $(ii)$ we find an explicit\linebreak $T_0(p)$ such that for any $t\geq T_0(p)$ the solution to the Fokker-Planck equation, $f(t)$, is in $L^2\pa{\R^d,f_\infty^{-1}}$. Proceeding similarly to the previous remark (but now with $q=2$ and $\epsilon=\frac{p-1}{4p}$) we have $\epsilon_1:=\min\pa{\frac{p-1}{4p},\frac14}=\frac{p-1}{4p}$. 
This yields the following upper bound for the ``waiting time'' in the hypercontractivity estimate \eqref{eq:entropichyper}:
$$
  T_0(p): =\frac{1}{\mu} \log\pa{\frac{\max\pa{c(5p-1),2c_2^2\frac{3p^2+p}{p+1}} \pa{1+\pa{\frac{2n}{\mu e}}^{2n}}}{p-1}}.
$$
Using Lemma \ref{lem:control_of_psi_by_e2}, Theorem \ref{thm:main} for $p=2$ (which was already proven in \S\ref{sec:spectral}), and inequality \eqref{eq:entropichyper} we conclude that for any $t\geq T_0(p)$
\begin{equation}\label{eq:main_proof_1p2_I}
\begin{gathered}
e_p(f(t)|f_\infty) \leq 2e_2(f(t)|f_\infty) \leq 2\tilde{c_2}e_2\pa{f(T_0(p))|f_\infty}\pa{1+\pa{t-T_0(p)}^{2n}}e^{-2\mu (t-T_0(p))}\\
\leq 2\tilde{c_p}e^{2\mu T_0(p)} \pa{p(p-1)e_p(f_0|f_\infty)+1}^{\frac{2}{p}}\pa{1+t^{2n}}e^{-2\mu t}.\quad
\end{gathered}
\end{equation}
To complete the proof we recall that any admissible relative entropy decreases along the flow of the Fokker-Planck equation (see \cite{AE} for instance). Thus, for any $t\leq T_0(p)$ we have that 
\begin{equation}\label{eq:main_proof_1p2_II}
e_p(f(t)|f_\infty) \leq e_p(f_0|f_\infty) \leq  e_p(f_0|f_\infty)e^{2\mu T_0(p)}\pa{1+t^{2n}}e^{-2\mu t}.
\end{equation}
The theorem now follows from \eqref{eq:main_proof_1p2_I} and \eqref{eq:main_proof_1p2_II}, together with the fact that for a $1<p<2$
$$e_p(f_0|f_\infty) \leq \mathcal{C}_p \pa{p(p-1)e_p(f_0|f_\infty)+1}^{\frac{2}{p}},$$
where $\mathcal{C}_p:=\sup_{x\geq 0}\frac{x}{(p(p-1)x+1)^{\frac{2}{p}}}<\infty.$
\end{proof}

We end this section with a slight generalization of our main theorem:
\begin{theorem}\label{thm:main_generalized}
Let $\psi$ be a generating function for an admissible relative entropy. Assume in addition that there exists $C_\psi>0$ such that
\begin{equation}\label{eq:psi_condition_hyper}
\psi_{p}(y) \leq C_{\psi}\psi(y) 
\end{equation}
for some $1<p<2$ and all $y\in \R^+$. Then, under the same setting of Theorem \ref{thm:main} (but now with the assumption $e_\psi(f_0|f_\infty)<\infty$) we have that 
\begin{equation*}%\label{eq:main_generalized}
e_\psi(f(t)|f_\infty) \leq c_{p,\psi} \pa{e_\psi(f_0|f_\infty)+1}^{\frac{2}{p}}\pa{1+t^{2n}}e^{-2\mu t},\quad t\ge0,
\end{equation*}
where $c_{p,\psi}>0$ is a fixed geometric constant.
\end{theorem}
\begin{proof}
The proof is almost identical to the proof of Theorem \ref{thm:main}. Due to \eqref{eq:psi_condition_hyper} we know that $e_p(f_0|f_\infty)<\infty$. As such, according to Theorem \ref{thm:hyper} $(ii)$ there exists an explicit $T_0(p)$ such that for all $t\geq T_0(p)$ we have that $f(t)\in L^2\pa{\R^d,f_\infty^{-1}}$ and 
$$e_2(f(t)|f_\infty)\leq \frac{1}{2}\pa{\pa{\frac{8\sqrt2}{3\cdot 2^\frac1p}}^{d}\pa{C_\psi p(p-1)e_\psi(f_0|f_\infty)+1)}^{\frac{2}{p}}-1}.$$
The above, together with Lemma \ref{lem:control_of_psi_by_e2} gives the appropriate decay estimate on $e_\psi$ for $t\geq T_0(p)$. Since $e_\psi$ decreases along the flow of our equation, we can deal with the interval $t\leq T_0(p)$ like in the previous proof, yielding the desired result.
\end{proof}

In the next, and last, section of this work we will mention another natural quantity in the theory of the Fokker-Planck equations - the Fisher information. We will briefly explain how the method we presented here is different to the usual technique one considers when dealing with the entropy. Moreover we describe how to infer from our main theorem an improved rate of convergence to equilibrium - in relative Fisher information.

\section{Decay of the Fisher Information}\label{sec:fisher}
The study of convergence to equilibrium for the Fokker-Planck equations via relative entropies has a long history. Unlike the study we presented here, which relies on detailed spectral investigation of the Fokker-Planck operator together with a non-symmetric hypercontractivity result, the common method to approach this problem - even in the degenerate case - is the so called \emph{entropy method}.\\
The idea behind the entropy method is fairly simple: once an entropy has been chosen and shown to be a Lyapunov functional to the equation, one attempts to find a linear relation between it and the absolute value of its dissipation. In the setting of the our equation, the latter quantity is referred to as \emph{the Fisher information}. \\
More precisely, it has been shown in \cite{AE} that:
\begin{lemma}\label{lem:entropydissipation}
Let $\psi$ be a generating function for an admissible relative entropy and let $f(t,x)$ be a solution to the Fokker-Planck equation \eqref{eq:fokkerplanck} with initial datum $f_0\in L^1_+\pa{\R^d}$. Then, for any $t>0$ we have that
\begin{equation*}%\label{eq:entropydissipation}
\begin{split}
\frac{d}{dt}&e_\psi\pa{f(t)|f_\infty} =\\
& -\int_{\R^d}\psi^{\prime\prime}\pa{\frac{f(t,x)}{f_\infty(x)}}\nabla\pa{ \frac{f(t,x)}{f_\infty(x)}}^T\C_s\nabla\pa{ \frac{f(t,x)}{f_\infty(x)}}f_{\infty}(x)dx \leq 0.
\end{split}
\end{equation*}
\end{lemma}
\begin{definition}
For a given positive semidefinite matrix ${\bf{P}}$ the expression 
\begin{equation*}%\label{def:fisher}
I_\psi^{\bf{P}}(f|f_\infty):=\int_{\R^d}\psi^{\prime\prime}\pa{\frac{f(x)}{f_\infty(x)}}\nabla\pa{\frac{f(x)}{f_\infty(x)}}^T {\bf{P}} \nabla\pa{\frac{f(x)}{f_\infty(x)}}f_\infty(x) dx\geq 0.
\end{equation*}
is called \emph{the relative Fisher Information generated by $\psi$}. 
\end{definition}
The entropy method boils down to proving that there exists a constant $\lambda>0$ such that 
\begin{equation}\label{eq:log_sob}
I_\psi^{\bf{P}}(f|f_\infty) \geq \lambda e_{\psi}(f|f_\infty).
\end{equation}
When $\D$ is positive definite, the above (with the choice $\bf{P}:=\D$) is a Sobolev inequality (and a log-Sobolev inequality for $\psi=\psi_1$), and a standard way to prove it is by using the Bakry-\'Emery technique (see \cite{AMTU01, BaEmD85} for instance). This technique involves differentiating the Fisher information along the flow of the Fokker-Planck equation and finding a closed functional inequality for it. By an appropriate integration in time, one can then obtain \eqref{eq:log_sob}.\\
Problems start arising with the above method when $\D$ is not invertible. As can be seen from the expression of $I_\psi^{\D}$ - there are some functions that are not identically $f_\infty$ yet yield a zero Fisher information. In recent work of Arnold and Erb (\cite{AE}), the authors managed to circumvent this difficulty by defining a new positive definite matrix ${\bf{P_0}}$ that is strongly connected to the drift matrix $\C$, and for which \eqref{eq:log_sob} is valid as a functional inequality. They proceeded to successfully use the Bakry-\'Emery method on $I_\psi^{{\bf{P_0}}}$ and conclude from it, and the log-Sobolev inequality, rates of decay for $I_\psi^{\D}$ (which is controlled by $I_\psi^{{\bf{P_0}}}$) and $e_\psi$. This is essentially what is behind the exponential decay in Theorem \ref{thm:Anton_Erb_rate}. Moreover, in the defective case (ii), it led to an $\epsilon$-reduced exponential decay rate. \\
As we have managed to obtain better convergence rates to equilibrium (in relative entropy) for the case of defective drift matrices $\C$, one might ask whether or not the same rates will be valid for the associated Fisher information $I_p^{\D}:=I_{\psi_p}^\D$. The answer to that question is \emph{Yes}, and we summarise this in the next theorem:
\begin{theorem}\label{thm:decay_for_fisher}
Consider the Fokker-Planck equation \eqref{eq:fokkerplanck} with diffusion and drift matrices $\D$ and $\C$ which satisfy Conditions (A)-(C). Let $\mu$ be defined as in \eqref{eq:def_of_mu} and assume that one, or more, of the eigenvalues of $\C$ with real part $\mu$ are defective. Denote by $n>0$ the maximal defect of these eigenvalues. Then, for any $1<p\leq 2$, the solution $f(t)$ to \eqref{eq:fokkerplanck} with initial datum $f_0\in L^1_+\pa{\R^d}$  that has a unit mass and 
%such that $e_p\pa{f_0|f_\infty}<\infty$, we have that if 
$I_p^{{\bf{P_0}}}(f_0|f_\infty)<\infty$ %for some positive definite matrix ${\bf{P}}$ 
satisfies:
\begin{equation*}%\label{eq:main_fisher}
I_p^{\D}\pa{f(t)|f_\infty}\le c I^{{\bf{P_0}}}_p\pa{f(t)|f_\infty} \leq 
c_p(f_0) \pa{1+t^{2n}}e^{-2\mu t},\quad t\ge0,
\end{equation*}
where $c_p(f_0)$ depends on $I_p^{{\bf{P}_0}}(f_0|f_\infty)$. %In particular, $I_p^{\D}$ decays like $e_p$.
\end{theorem}
\begin{proof}
We first note that Proposition 4.4 from \cite{AE} implies the estimate $e_p\pa{f_0|f_\infty}\le c I_p^{{\bf{P_0}}}(f_0|f_\infty)$ $<\infty$, and hence Theorem \ref{thm:main} applies. This decay of $e_p$ carries over to $I_p^{{\bf{P_0}}}$ due to the following two ingredients:
For small $t$ we can use the purely exponential decay of $I_p^{{\bf{P_0}}}$ as established in Proposition 4.5 of \cite{AE} (with the rate $2(\mu-\epsilon)$). And for large time we use the (degenerate) parabolic regularisation of the Fokker-Planck equation \eqref{eq:fokkerplanck}: As proven in Theorem 4.8 of \cite{AE} we have for all $\tau\in (0,1]$ that
\begin{equation*}%\label{eq:regularization}
I^{\bf{P}_0}_\psi(f(\tau)|f_\infty) \leq \frac{c_{k_0}}{\tau^{2\kappa+1}}e_{\psi}\pa{f_0|f_\infty},
\end{equation*} 
where $\psi$ is the generating function for an admissible relative entropy. And $\kappa>0$ is the minimal number such that there exists $\tilde{\lambda}>0$ with 
$$\sum_{j=0}^{\kappa} {\C}^j \D \pa{\C^T}^j \geq \tilde{\lambda} \bf{I}. $$
The existence of such $\kappa$ and $\tilde{\lambda}$ is guaranteed by Condition (C) and equivalent to the rank condition \eqref{rank-cond}- cf.\ Lemma 2.3 in \cite{AAS}.
\end{proof}

\bibliographystyle{plain}

\begin{thebibliography}{99}

\bibitem{AAS} 
F. Achleitner, A. Arnold, D. St{\"{u}}rzer, \textit{Large-Time Behavior in Non-Symmetric Fokker-Planck Equations}.
Rivista di Matematica della Universit\`a di Parma \textbf{6} (2015), 1--68. 
%Preprint. https://arxiv.org/abs/1506.02470
%\newblock x, 2014.

\bibitem{AE}
A. Arnold, J. Erb, \textit{Sharp Entropy Decay for Hypocoercive and Non-Symmetric Fokker-Planck Equations with Linear Drift}. Preprint. https://arxiv.org/abs/1409.5425 .

%\bibitem{ArCaJuL08}
%A. Arnold, E. Carlen, Q. Ju, \textit{Large-time behavior of non-symmetric Fokker-Planck type equations}, Communications in Stochastic Analysis \textbf{2} (2008) 153--175.

%\bibitem{ArCaMa10}
%A. Arnold,  J.A. Carrillo, C. Manzini, \textit{Refined long-time asymptotics for some polymeric fluid flow models}, Comm. Math. Sc. \textbf{8} (2010), 763--782.  

%\bibitem{ArCaMa17}
%A. Arnold, J.A. Carillo, D. Matthes, Preprint (2017)

\bibitem{AMTU01}
A. Arnold, P. Markowich, G. Toscani, A. Unterreiter, \textit{On convex Sobolev inequalities and the rate of convergence to equilibrium for Fokker--Planck type equations}, Communications in Partial Differential Equations \textbf{26} (2001), 43--100.

\bibitem{BaGeLe14}
D. Bakry, I. Gentil, M. Ledoux, \textit{Analysis and Geometry of Markov Diffusion Operators}, Springer (2014).

\bibitem{BaEmD85}
D. Bakry, M. \'Emery, \textit{Diffusions hypercontractives}, S\'eminaire de probabilt\'es de Strasbourg \textbf{19} (1985), 177--206.

%\bibitem{BaEmH84}
%D. Bakry, M. \'Emery, \textit{Hypercontractivit\'e de semi-groupes de diffusion}, C. R. Acad. Sci. Paris S\'er. I Math. \textbf{299} (1984), 775--778.

%\bibitem{BaEmI85}
%D. Bakry, M. \'Emery, \textit{In\'egalit\'es de Sobolev pour un semi-group sym\'etrique}, C. R. Acad. Sci. Paris S\'er. I Math. \textbf{301} (1985), 411--413.

%\bibitem{BaB13}
%F. Baudoin, \textit{Bakry-\'Emery meet Villani}, http://arxiv.org/abs/1308.4938
\bibitem{BWY15SPA}
J. Bao, F.-Y. Wang, C. Yuan, \textit{Hypercontractivity for functional stochastic differential equations}, Stoch. Proc. Appl.125 (2015), 3636--3656.

\bibitem{BWY15EJP}
J. Bao, F.-Y. Wang, C. Yuan, \textit{Hypercontractivity for Functional Stochastic Partial Differential Equations}, Electron. J.  Probab. \textbf{20} (2015), no. 93, 15 pp.

%\bibitem{BoGe10}
%F. Bolley, I. Gentil, \textit{Phi-entropy inequalities for diffusion semigroups}, J. Math. Pures Appl. \textbf{93} (2010), 449--473.

%\bibitem{DeH06}
%L. Desvillettes, \textit{Hypocoercivity: The example of linear transport}, Contemporary Mathematics \textbf{409}: Recent Trends in Partial Differential Equations (2006), 33--53.

%\bibitem{DeVi01}
%L. Desvillettes, C. Villani, \textit{On the trend to global equilibrium in spatially inhomogeneous
%entropy-dissipating systems: The linear Fokker-Planck equation}, Comm. Pure Appl. Math. \textbf{54} (2001), 1--42.

%\bibitem{DoMoScH10}
%J. Dolbeault, C. Mouhot, C. Schmeiser, \textit{Hypocoercivity for linear kinetic equations conserving mass}, Trans.\ Amer.\ Math.\ Soc.\ 367 (2015), no. 6, 3807--3828. 

%\bibitem{DuH09}
%R. Duan, \textit{Hypocoercivity of Linear Degenerately Dissipative Kinetic Equations}, Nonlinearity \textbf{24} (2010), 2165--2189.

%\bibitem{Erb14}
%J. Erb, \textit{Entropy Method for Hypocoercive Fokker-Planck Type Equations}, PhD-thesis, Vienna Univ. of Technology, 2014.

%\bibitem{EnNaS06}
%K.-J. Engel, R. Nagel, \textit{A Short Course on Operator Semigroups}, Springer 2006.

\bibitem{EN00}
K.-J. Engel, R. Nagel, \textit{One-Parameter Semigroups for Linear Evolution Equations}, Springer 2000.

%\bibitem{FMP06}
%F. Filbet, C. Mouhot, L. Pareschi, \textit{Solving the Boltzmann equation in $N\,\log_2\,N$}, SIAM J. Sci. Comput. \textbf{28} (2006), 1029--1053.

\bibitem{F98}
M. Fuhrman, \textit{Hypercontractivity properties of nonsymmetric Ornstein-Uhlenbeck semigroups in Hilbert
spaces}, Stochastic Anal. Appl. \textbf{16} (1998), no. 2, 241-260.

%\bibitem{GaMiS13}
%S. Gadat, L. Miclo, \textit{Spectral decompositions and $L^2$-operator norms of toy hypocoercive semi-groups}, Kinetic and Related Models \textbf{2} (2013), 317--372.

%\bibitem{GKZ04}
%A.N. Gorban, I.V. Karlin, A.Yu. Zinovyev, \textit{Constructive methods of invariant manifolds for kinetic problems}, Phys. Rep.  \textbf{396} (2004), 197--403.

\bibitem{Gro75}
L. Gross, \textit{Logarithmic Sobolev inequalities}, Amer. J. Math.  \textbf{97} (1975), no. 4, 1061--1083. 


%\bibitem{GuIWS88}
%P. Gurka, \emph{$A_r$-condition for two weight functions and compact imbeddings of weighted Sobolev spaces}, Czechoslovak Mathematical Journal \textbf{38} (1988), 611--617.

%\bibitem{HeNiFP05}
%B. Helffer, F. Nier, \textit{Hypoelliptic estimates and spectral theory for Fokker-Planck operators and Witten Laplacians}, Lecture Notes in Mathematics 1862 (2005).

\bibitem{HS}
B. Helffer and J. Sj\"ostrand.
\textit{From resolvent bounds to semigroup bounds}. Preprint. ArXiv: 1001.4171v1.

%\bibitem{HeF07}
%F. H\'erau, \textit{Short and long time behavior of the Fokker-Planck equation in a confining potential and applications}, Journal of Functional Analysis \textbf{244} (2007), 95--118.

%\bibitem{HeNi04}
%F. H\'erau, F. Nier, \textit{Isotropic Hypoellipticity and Trend to
%Equilibrium for the Fokker-Planck Equation
%with a High-Degree Potential}, Arch. Rational Mech. Anal. \textbf{171} (2004), 151--218.

%\bibitem{Hi70}
%C.D. Hill, \textit{A sharp maximum principle for degenerate elliptic-parabolic equations}, Indiana Univ. Math. J. \textbf{20} (1970), 213--229.

\bibitem{Ho67}
L. H\"ormander, \textit{Hypoelliptic second order differential equations}, Acta Math. \textbf{119} (1969), 147--171.

%\bibitem{Ho95}
%L. H\"ormander, \textit{Symplectic classification of quadratic forms, and general Mehler formulas}, Math. Z. \textbf{219} (1995), 413--449.

\bibitem{HoJoT91}
R.A. Horn, C.R. Johnson, \textit{Topics in Matrix Analysis}, Cambridge University Press (1991).

%\bibitem{KoH73}
%J.J. Kohn, \textit{pseudo-differential operators and hypoellipticity}, Proc. Symp. Pure Math. \textbf{23} (1973), 61--69.

\bibitem{LV13}
Y. Latuskhin, Y. Valerian, \textit{Stability estimates for semigroups on Banach spaces}, Discrete Contin. Dyn. Syst. \textbf{33}, no. 11-12 (2013), 5203--5216.

%\bibitem{LNP13}
%T. Leli\`evre, F. Nier, G.A. Pavliotis, \textit{Optimal Non-reversible Linear Drift for the Convergence
%to Equilibrium of a Diffusion}, J. Stat. Phys. \textbf{152} (2013), 237--274.

%\bibitem{MPP02}
%G. Metafune, D. Pallara, E. Priola, \textit{Spectrum of Ornstein-Uhlenbeck operators in $L^p$ spaces
%with respect to invariant measures}, J. Funct. Anal. \textbf{196} (2002), 40--60.

%\bibitem{MoH14}
%P. Monmarch\'e, \textit{Hypocoercive relaxation to equilibrium for some kinetic models}, Kinetic and Related Models \textbf{7} (2014), 341--360.

\bibitem{MoH15Arx}
P. Monmarch\'e, \textit{Generalized $\Gamma$ calculus and application to interacting particles on a graph}, Preprint. https://arxiv.org/abs/1510.05936


%\bibitem{MoNeQ06}
%C. Mouhot, L. Neumann, \textit{Quantitative perturbative study of convergence to equilibrium for collisional kinetic models in the torus}, Nonlinearity \textbf{19} (2006), 969--998.

\bibitem{Ne73}
E. Nelson, \textit{The free Markov field}, J. Funct. Anal., \textbf{12} (1973), 211--227. 
%MR-0343816

%\bibitem{Ni14}
%F. Nier, \textit{Accurate Estimates for the Exponential Decay of Semigroups with Non-Self-Adjoint Generators}; in: 
%O.N. Kirillov, D.E. Pelinovsky (eds.), \textit{Nonlinear Physical Systems: Spectral Analysis, Stability and Bifurcations}, ISTE \& Wiley (2014).

\bibitem{OPPS12}
M. Ottobre, G.A. Pavliotis, K. Pravda-Starov, \textit{Exponential return to equilibrium for hypoelliptic
quadratic systems}, J. Funct. Anal. \textbf{262} (2012), 4000--4039.

\bibitem{OPPS15}
M. Ottobre, G.A. Pavliotis, K. Pravda-Starov, \textit{Some remarks on degenerate hypoelliptic Ornstein-Uhlenbeck
operators}, J. Math. Anal. Appl. \textbf{429} (2015), 676--712.

%\bibitem{PaS83}
%A. Pazy, \textit{Semigroups of Linear Operators and Applications to Partial Differential Equations}, Springer (1983).

\bibitem{RiFP89}
H. Risken, \textit{The Fokker-Planck equation. Methods of solution and applications.}, Springer-Verlag (1989).

%\bibitem{RoStH77}
%L.P. Rothschild, E.M. Stein, \textit{Hypoelliptic differential operators and nilpotent groups}, Acta Math. \textbf{137} (1976), 247--320.

\bibitem{SnZaN70}
J. Snyders, M. Zakai, \textit{On nonnegative solutions of the equation $AD+DA'=-C$}, SIAM J. Appl. Math. \textbf{18} (1970), 704--715.

%\bibitem{CaDoMaSpF04}
%C. Sparber, J.A. Carrillo, J. Dolbeault, P. Markowich,  \textit{On the long-term behaviour of the Quantum Fokker--Planck equation}, Monatshefte f\"ur Mathematik \textbf{141} (2004), 237--257.

%\bibitem{TaFPI02}
%D. Tataru, \textit{On The Fefferman-Phong Inequality And Related Problems}, Communications in Partial Differential Equations \textbf{26} (2002), 2101--2138.

%\bibitem{Te09}
%G. Teschl, \textit{Mathematical Methods in Quantum Mechanics with Applications to Schr\"odinger Operators}, Graduate Studies in Mathematics, volume 157, AMS (2009) 

%\bibitem{ViH06}
%C. Villani, \textit{Hypocoercivity}, Memoirs of the American Mathematical Society \textbf{202} (2009).

\bibitem{W17}
F-Y. Wang. \textit{Hypercontractivity and applications for stochastic Hamiltonian systems}, J. Funct. Anal. \textbf{272}, no. 12 (2017), 5360--5383.

\bibitem{W05}
F-Y. Wang. \textit{Functional Inequalities, Markov Semigroups and Spectral Theory}, Science-Press (2005).

%\bibitem{Vi02}
%C. Villani, \textit{A review of mathematical topics in collisional kinetic theory}. Handbook
%of Fluid Mechanics, Vol. \textbf{1}, S. Friedlander and D. Serre (eds.), North-Holland (2002).

\end{thebibliography}
\end{document}